\documentclass[12pt]{amsart}
\usepackage[all]{xy}
\usepackage[parfill]{parskip}

\usepackage{amsfonts}
\usepackage{mathrsfs}
\usepackage[abs]{overpic}

\usepackage{palatino} 
\usepackage[latin1]{inputenc}
\usepackage{hyperref}

\usepackage{verbatim}
\usepackage{color}

\usepackage{amsmath, amscd, graphicx, latexsym, hyperref, times, calc}
\usepackage{color,graphicx}
\usepackage[abs]{overpic}
\usepackage{tikz}
\usepackage{tikz-cd}
\usetikzlibrary{arrows, patterns}

\newtheorem{dummy}{anything}[section]
\newtheorem{theorem}[dummy]{Theorem}

\newtheorem{lemma}[dummy]{Lemma}

\newtheorem{question}[dummy]{Question}

\newtheorem{proposition}[dummy]{Proposition}

\newtheorem{conjecture}[dummy]{Conjecture}

\theoremstyle{definition}
\newtheorem{definition}[dummy]{Definition}

\newtheorem{example}[dummy]{Example}

\newtheorem{remark}[dummy]{Remark}

\newtheorem*{acknowledgements}{Acknowledgements}

\begin{document}

\title{Nonexistence and existence of fillable contact structures on 3-manifolds}

\author{Fan Ding, Youlin Li and Zhongtao Wu}

\address{School of Mathematical Sciences and LMAM, Peking University, Beijing 100871, China}
\email{dingfan@math.pku.edu.cn}

\address{School of Mathematical Sciences, Shanghai Jiao Tong University, Shanghai 200240, China}
\email{liyoulin@sjtu.edu.cn}

\address{Department of Mathematics, The Chinese University of Hong Kong, Shatin, Hong Kong}
\email{ztwu@math.cuhk.edu.hk}

\subjclass[2000]{}

\begin{abstract}
In the first part of this paper, we construct infinitely many hyperbolic closed 3-manifolds which admit no symplectic fillable contact structure.  All these 3-manifolds are obtained by Dehn surgeries along L-space knots or L-space two-component links. In the second part of this paper, we show that Dehn surgeries along certain knots and links, including those considered in the first part, admit Stein fillable contact structures as long as the surgery coefficients are sufficiently large. This provides some new evidence for the high surgery conjecture raised by Stipsicz.
\end{abstract}

\maketitle

\section{Introduction}\label{sec: intro}
A contact 3-manifold is either tight or overtwisted. Given a contact 3-manifold $(Y,\xi)$, it is a fundamental question to ask whether it is tight or overtwisted. If a contact 3-manifold is weakly symplectic fillable, then it is tight. If a contact 3-manifold $(Y,\xi)$ is Stein fillable, then it is strongly symplectic fillable, and hence weakly symplectic fillable. The reader may refer to \cite{ge} for these definitions.

In this paper, we are concerned with the existence and non-existence of the various kind of contact structures on 3-manifolds.
\begin{question}\label{Q1}
Given an irreducible closed oriented 3-manifold,\\
(1) does it admit a tight contact structure? \\
(2) does it admit a (weakly or strongly) symplectic fillable contact structure? \\
(3) does it admit a Stein fillable contact structure?
\end{question}

Suppose $Y$ is an irreducible closed oriented 3-manifold. If $b_{1}(Y)>0$, then $Y$ admits a taut foliation (cf. \cite{et}). By \cite[Corollary 3.2.5]{et}, the taut foliation can be perturbed to a weakly symplectic semi-fillable, and hence a weakly symplectic fillable contact structure by gluing some symplectic caps \cite{e2,e, ge1}. If $b_{1}(Y)=0$, recall that a rational homology sphere is called an L-space if the rank of its Heegaard Floer homology $\widehat{HF}(Y)$ equals to the order of $H_{1}(Y)$.  If $Y$ is not an L-space, the L-space conjecture in \cite{BGW} says that $Y$ admits a taut foliation, and hence a weakly symplectic fillable contact structure. So the existence of (weakly or strongly) symplectic fillable contact structures on L-space is worth investigating. 

According to \cite{e0} and \cite{OhOn}, for a rational homology sphere, the existence of a weakly symplectic fillable contact structure is equivalent to the existence of a strongly symplectic fillable one. So, we only consider the existence of weakly symplectic fillable contact structure on a rational homology sphere. In the sequel a symplectic fillable contact structure always stand for a weakly symplectic fillable contact structure unless specified otherwise.

In \cite{LeLi}, Lecuona and Lisca classified the Seifert fibered spaces which admit no fillable contact structure. In particular, they show that a Seifert fibered space admits a symplectic fillable contact structure if and only if it admits a Stein fillable contact structure.

In \cite{KaTo} and \cite{liliu}, some hyperbolic Dehn surgeries along the pretzel knot $P(-2,3,2n+1)$, $n\geq3$, are found to admit no symplectic fillable contact structures. In this paper, we consider Dehn surgeries along closed 3-braids which are L-space knots.  Recall that a knot $K$ in $S^3$ is an L-space knot if some positive Dehn surgery along $K$ yields an L-space.

A twisted torus knot $T(u,v;u',v')$ is a knot which can be obtained by adding $v'$ full twists along $u'$ adjacent strands to the torus knot $T(u,v)$, where $2\leq u'\leq u-1$. Suppose $n\geq 2$ and $m\geq 1$. Let $K_{n,m}$ be the twisted torus knot $T(3,3m+2;2,n-2)$, and $K'_{n,m}$ be the twisted torus knot $T(3,3m+1;2,n-2)$. See Figure~\ref{figure:Knm}. Note that  $K_{n,1}$ is the pretzel knot $P(-2,3,2n+1)$ \cite{t}, while both $K_{2,m}$ and $K'_{2,m}$ are positive torus knots. By the main result of \cite{LeeVa}, a closed 3-braid which is an L-space knot must have the type of $K_{n,m}$ or $K'_{n,m}$.

\begin{figure}[htb]
\begin{overpic}
{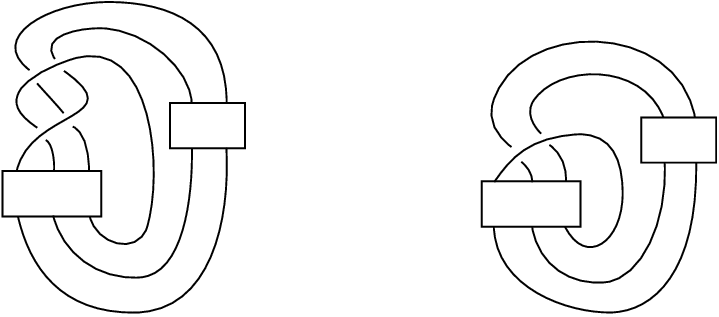}
\put(20, 57){$m$}
\put(86, 89){$n-2$}
\put(250, 52){$m$}
\put(314, 82){$n-2$}
\end{overpic}
\caption{The left picture is the twisted torus knot $K_{n,m}$. The right picture is the twisted torus knot $K'_{n,m}$. The boxes represent $m$ full right handed twists and $n-2$ full right handed twists.}
\label{figure:Knm}
\end{figure}
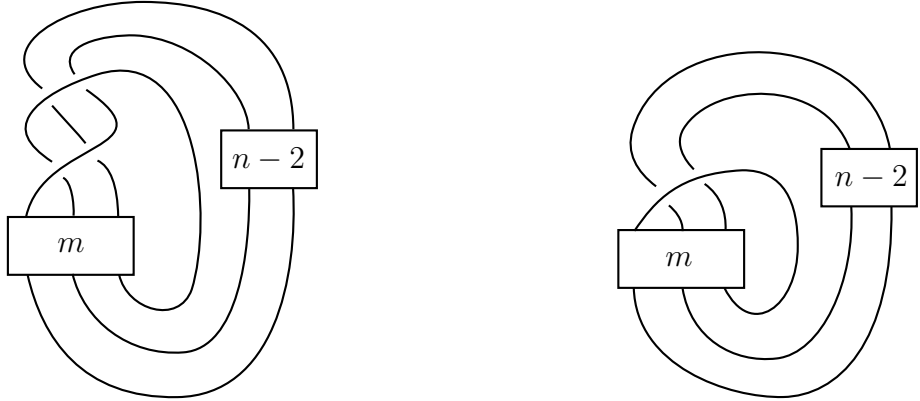

Our first main result concerns surgeries on these braid-index-three L-space knots that admit no symplectic fillable contact structures. It is an extension of the results in \cite{KaTo} and \cite{liliu}.

\begin{theorem}\label{Theorem:Main1}
Suppose $K$ is a closed 3-braid which is an L-space knot, then the rational  $r$-surgery  along $K$ yields a 3-manifold admitting no symplectic fillable contact structure for $r\in[2g(K)-1, 2g(K)]$, where $g(K)$ is the genus of $K$. Precisely speaking, the 3-manifold $S^{3}_{r}(K_{n,m})$ (resp. $S^{3}_{r}(K'_{n,m})$) admits no symplectic fillable contact structures for $r\in[2n+6m-3, 2n+6m-2]$ (resp. $r\in[2n+6m-5, 2n+6m-4]$).
\end{theorem}


\begin{remark}
By \cite[Corollary 1.2]{Lee}, both $K_{n,m}$ and $K'_{n,m}$ are hyperbolic knots whenever $n\geq4$. So if $n\geq4$ then $S^{3}_{r}(K_{n,m})$ (resp. $S^{3}_{r}(K'_{n,m})$) is a hyperbolic 3-manifold for all but finitely many $r\in[2n+6m-3, 2n+6m-2]$ (resp. $r\in[2n+6m-5, 2n+6m-4]$). The genus of $K_{n,m}$ is $n+3m-1$.  The maximal Thurston-Bennequin invariant of  $K_{n,m}$ is $2n+6m-3$ which equals to $2g(K_{n,m})-1$. The genus of $K'_{n,m}$ is $n+3m-2$. The maximal Thurston-Bennequin invariant of  $K'_{n,m}$ is $2n+6m-5$ which equals to $2g(K'_{n,m})-1$. By the main result in \cite{ls3},  if $K$ is either $K_{n,m}$ or $K'_{n,m}$, and $r>2g(K)-1$, then  $S^{3}_{r}(K)$ admits tight contact structures. However, to the authors' knowledge, it is still unknown whether the $(2g(K)-1)$-surgery along such a knot $K$ admits a tight contact structure.
\end{remark}


In fact, for any L-space knot $K\subset S^3$, one can use the approach in this paper and likely find more surgeries that yield 3-manifolds admitting no symplectic fillable contact structure.  Here we give a sufficient condition for the $(2g-1)$-surgery.

\begin{theorem}\label{Theorem:Main2}
Suppose $K\subset S^3$ is an L-space knot and $g(K)=g$. For $0\leq k\leq [\frac{g-1}{4}]+1$ and $k\in \mathbb{Z}$, let $$i_{k}=\min\{i\in\mathbb{Z}_{\geq0}\mid 2i> 2g-1-\sqrt{(8k+1)(2g-1)} \}.$$  If the torsion coefficients (defined by the formula (\ref{t_i})) $t_{i_{k}}(K)\geq k+1$ for all integers $k\in [0, [\frac{g-1}{4}]+1]$, then the manifold $S^{3}_{2g-1}(K)$ admits no symplectic fillable contact structure.
\end{theorem}

Now we consider Dehn surgeries along two-component L-space links.  Recall that a two-component link in $S^3$ is an L-space link if all of its positive large Dehn surgeries yield L-spaces \cite[Definition 1.8]{liu}.

Let $\mathbb{L}_n$, $n\geq 1$, be the link shown in Figure~\ref{figure:Ln}, where the first component $L_1$ is the unknot, and the second component $L_2$ is the torus knot $T(2,2n+1)$.

\begin{theorem}\label{Theorem:Main3}
The $(p_1, p_2)$-surgery along $\mathbb{L}_{n}$ yields a 3-manifold which admits no symplectic fillable contact structure, where \\
1, $p_{1}\in\{2, 3, 4\}$, $p_{2}=2n+1$, or \\
2, $p_{1}=5$, $p_{2}=2n+1$ and $5<2n+1$, or \\
3, $p_{1}=2$, $p_{2}=2n+2$ and $n+1$ is not a square.
\end{theorem}

\begin{figure}[htb]
\begin{overpic}
{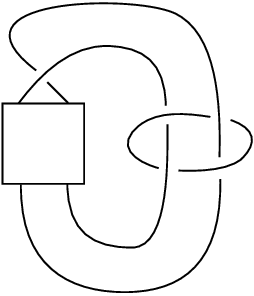}

\put(120, 86){$p_1$ }
\put(-10, 128){$p_2$ }
\put(20, 70){$n$ }
\end{overpic}
\caption{ Two-component link $\mathbb{L}_n$. The box represents $n$ full right handed twists.}
\label{figure:Ln}
\end{figure}

\begin{remark} In the Thistlethwaite Link Table,  $\mathbb{L}_{1}$ is  $L7a3$,  $\mathbb{L}_{2}$ is $L9a14$, and $\mathbb{L}_{3}$ is $L11a110$. By verification from SnapPy, the $(2, 3)$, $(2,4)$, $(3,3)$-surgeries along $\mathbb{L}_{1}$,  the $(2, 5)$, $(3,5)$, $(4,5)$, $(2,6)$-surgeries along $\mathbb{L}_{2}$, and the $(2,7)$, $(3,7)$, $(4,7)$, $(5,7)$, $(2,8)$-surgeries along $\mathbb{L}_{3}$ all yield hyperbolic 3-manifolds.
\end{remark}


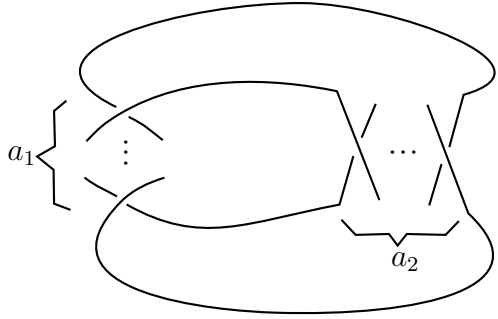
\begin{figure}[htb]
\begin{overpic}
{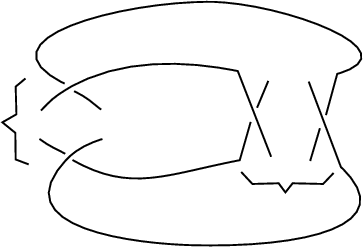}
\put(33, 60){$\cdot$}
\put(33, 64){$\cdot$}
\put(33, 56){$\cdot$}
\put(134, 60){$\cdot$}
\put(138, 60){$\cdot$}
\put(142, 60){$\cdot$}
\put(-10, 60){$a_1$}
\put(135, 20){$a_2$}
\end{overpic}
\caption{Two-bridge link $K(a_1,a_2)$. For $i=1,2$, if $a_i\geq0$, then there are $a_i$ positive crossings. If $a_i<0$, then there are $-a_i$ negative crossings. The link $K(a_1,a_2)$ is a knot if and only if either $a_1$ or $a_2$ is even. }
\label{figure:simple2bridge0}
\end{figure}

Let $K(a_{1}, a_{2})$ be the two-bridge link shown in Figure~\ref{figure:simple2bridge0},  where both $a_{1}$ and $a_{2}$ are positive odd integers. By \cite[Theorem 3.8]{liu}, $K(a_{1}, a_{2})$ is an L-space link if both $a_{1}$ and $a_{2}$ are positive odd integers.  Note that the two-bridge link $K(a_{1}, a_{2})$ is denoted by $b(a_{1}a_{2}-1, -a_{1})$ in \cite{liu} and \cite{glm}.


\begin{proposition}\label{proposition:K(5,5)}
The $(3,3)$-surgery along the two-bridge link $K(5, 5)$ yields a 3-manifold which admits no symplectic fillable contact structure.
\end{proposition}

\begin{remark}
In the Thistlethwaite Link Table,  $K(5,5)$ is  the mirror image of $L9a40$. Verified by SnapPy, $S^{3}_{3,3}(K(5, 5))$ is a hyperbolic 3-manifold. 
\end{remark}

As all known examples of irreducible closed 3-manifolds which admit no symplectic fillable contact structures are L-spaces, the following questions may be interesting.

\begin{question}
Is there an irreducible rational homology sphere which is not an L-space and admits no Stein fillable contact structure?
\end{question}

\begin{question}
Is there a closed irreducible 3-manifold which is not a rational homology sphere and admits no Stein fillable contact structure, or no strongly symplectic fillable contact structure?
\end{question}

There is a potential approach to finding an irreducible rational homology sphere which is not an L-space and admits no Stein fillable contact structure.

\begin{proposition}\label{proposition:Main5}
Suppose $Y$ is a rational homology sphere whose correction terms are all negative, and $HF^{+}(Y,\mathfrak{s})$ $=\mathcal{T}^{+}\oplus \mathbb{F}$ for a self-conjugate Spin$^{c}$ structure $\mathfrak{s}\in Spin^{c}(Y)$. If\\
(1) $d(Y,\mathfrak{s})<-2.5$, or \\
(2) $d(Y,\mathfrak{s})<-2.25$ and $HF^{+}(Y)=\mathcal{T}^{+}_{d}\oplus \mathbb{F}_{d}$, or\\ (3) $d(Y,\mathfrak{s})<0$ and $\alpha(Y,\mathfrak{s})<\frac{d(Y,\mathfrak{s})}{2}+1$, \\where $d(Y,\mathfrak{s})$ is the correction term, $\alpha(Y,\mathfrak{s})$ is a Manolescu invariant,
then $Y$ admits no Stein fillable contact structures whose underlying Spin$^{c}$ structure is $\mathfrak{s}$.
\end{proposition}

In the special case where $Y$ is an integer homology sphere, the unique Spin$^{c}$ structure $\mathfrak{s}$ on $Y$ is self-conjugate. We denote $d(Y,\mathfrak{s})$ by $d(Y)$, and $\alpha(Y,\mathfrak{s})$ by $\alpha(Y)$. It is known that $\alpha(Y)$ is an integer, $d(Y)$ is an even integer, and $\frac{1}{2}d(Y)\leq \alpha(Y)$ \cite[Theorem1.2]{st}. Suppose furthermore that $HF^{+}(Y)=\mathcal{T}^{+}\oplus \mathbb{F}$. According to Proposition~\ref{proposition:Main5}, if $Y$ satisfies either $d(Y)<-2$ or $d(Y)=2\alpha(Y)=-2$, then $Y$ admits no Stein fillable contact structures.  

\begin{question}
Is there an irreducible integer homology sphere $Y$ which satisfies \\
(1) $HF^{+}(Y)=\mathcal{T}^{+}\oplus \mathbb{F}$, and \\
(2) either $d(Y)<-2$; or $d(Y)=2\alpha(Y)=-2$?
\end{question}

\bigskip

On the opposite side of the existence problem, it is well-known that $S^{3}_{r}(K)$ admits a Stein fillable contact structure if the rational number $r$ is smaller than the maximal Thurston-Bennequin invariant $TB(K)$ of a knot $K$ in $S^3$.  While Theorem~\ref{Theorem:Main1} tells us that certain surgeries along $K$ may admit no symplectic fillable contact structure, we will see that the 3-manifolds resulting from sufficiently large surgeries along many knots, including the ones in Theorem~\ref{Theorem:Main1}, do admit Stein fillable contact structures, or equivalently, bound Stein domains.



\begin{theorem}\label{thm:closed3braids}
Suppose $K$ is a closed 3-braid which is an L-space knot, or the connected sum of any number of such knots. Then the $r$-surgery along $K$ yields a 3-manifold which admits a Stein fillable contact structure if the rational $r$ is sufficiently large. In particular, the 3-manifold $S^{3}_{r}(K_{n,m})$ (resp. $S^{3}_{r}(K'_{n,m})$) admits a Stein fillable contact structure if $r\geq 9m+4n-4$ (resp. if $r\geq 9m+4n-8$).
\end{theorem}


\begin{theorem}\label{thm:2bridgeknot}
Let $K$ be a two-bridge knot $K(a_1,a_2)$ of the form shown in Figure~\ref{figure:simple2bridge0}, then $S^3_{r}(K)$ admits a Stein fillable contact structure if the rational $r$ is sufficiently large.
\end{theorem}

\begin{theorem}\label{thm:torusknots}
Suppose $K$ is a torus knot or the connected sum of any number of torus knots. Then $S^3_{r}(K)$ admits a Stein fillable contact structure if the rational $r$ is sufficiently large.
\end{theorem}

\begin{theorem}\label{thm:pretzel}
Suppose $K$ is the pretzel knot $P(-2l-1,2m+1,2n+1)$, where $l,m,n$ are positive integers. Then $S^3_{r}(K)$ admits a Stein fillable contact structure if the rational $r$ is sufficiently large.
\end{theorem}

So we formulate the following question.

\begin{question}\label{ques:highsurgery}
Suppose $K$ is a knot in $S^3$. Does $S^{3}_{r}(K)$ admit a Stein fillable contact structure for $r$ sufficiently large?
\end{question}

This question is closely related to the following conjecture raised by Stipsicz \cite{s} in 2010.

\begin{conjecture}[High surgery conjecture]\label{conj:highsurgery}
Suppose $K\subset S^3$ is a knot. Then there is an integer $n_K$ such that for all $r\geq n_K$ the surgered 3-manifold $S^{3}_{r}(K)$ admits a tight contact structure.
\end{conjecture}

If the answer to Question~\ref{ques:highsurgery} is yes, then Conjecture~\ref{conj:highsurgery} is true, but not vice-versa. It is known that some knots in $S^3$ satisfy the high surgery conjecture, see \cite{ls3, g1, mt, hp, c, cm}. However, as we will see in Section~\ref{sec:existence}, one cannot conclude that Conjecture~\ref{conj:highsurgery} holds for the knots considered in Theorems~\ref{thm:2bridgeknot}, \ref{thm:torusknots} and \ref{thm:pretzel} from the previously known results. 

In view of this phenomenon, we propose to define an invariant for knots in terms of the Stein fillability of surgeries on knots.

\begin{definition}
The {\it Stein fillable coefficient} of a knot $K$ in $S^3$ is defined as $$\inf\{r\in\mathbb{Q}\cup\{\infty\} \mid r\geq TB(K), \text{and}~S^{3}_{r'}(K)~\text{bounds a Stein domain for any}~r'\geq r\}.$$
\end{definition}
We denote the Stein fillable coefficient of $K$ by $Sfc(K)$.

For a knot $K$ in $S^3$, $m(K)$ is an invariant of $K$ introduced by Owens and Strle \cite{OwSt1}, and is defined as $$m(K)=\inf\{r\in \mathbb{Q}_{>0}\mid S^{3}_{r}(K)~\text{bounds a negative-definite 4-manifold}\}.$$ In some cases, $m(K)$ provides a lower bound on the Stein fillable coefficients.

\begin{proposition}\label{sfcandm}
Suppose $K\subset S^3$ is an L-space knot and $TB(K)=2g(K)-1$. Then $Sfc(K)\geq m(K)$.
\end{proposition}

Note that Lidman and Sivek conjecture that $TB(K)=2g(K)-1$ holds for any L-space knot $K$ \cite{LiSi}.

More generally, we formulate the following bolder question for links.

\begin{question}\label{ques:link}
Suppose $L=K_{1}\cup K_{2}\cdots\cup K_{n}$ is a link in $S^3$. Does the $(r_{1}, r_{2},\cdots, r_{n})$-surgery along $L$ yields a 3-manifold which admits a Stein fillable contact structure for the rational $r_{i}$ sufficiently large? where $i=1,2,\cdots,n$.
\end{question}

It is not hard to show that this question is true for Hopf link,  Whitehead link and Borromean ring. Moreover, we answer this question for the link $\mathbb{L}_n$ appeared in Theorem~\ref{Theorem:Main3} and the two-bridge links $K(a_1,a_2)$ with two components shown in Figure~\ref{figure:simple2bridge0}.

\begin{theorem}\label{thm:biglink}
The $(r_1, r_2)$-surgery along the link $\mathbb{L}_n$ shown in Figure~\ref{figure:Ln} yields a 3-manifold which admits a Stein fillable contact structure if the rationals $r_1>0$ and $r_2>4n+4$.
\end{theorem}

\begin{theorem}\label{thm:biglink1}
Let $\mathbb{L}$ be a two-bridge link $K(a_1,a_2)$ with two components in $S^3$ of the form shown in Figure~\ref{figure:simple2bridge0}, then $S^3_{r_1, r_2}(\mathbb{L})$ admits a Stein fillable contact structure if both the rationals $r_1$ and $r_2$ are sufficiently large.
\end{theorem}

\medskip
\begin{acknowledgements}
The authors would like to thank John Etnyre, Matt Hedden, Jianfeng Lin, Beibei Liu, B\"ulent Tosun and Xiao Wang for some correspondence. We are also grateful to the
referees for valuable suggestions.
The first author was partially supported by National Key R\&D Program of China (No. 2020YFA0712800) and Grant No. 12131009 of the National Natural Science Foundation of China. The second author was partially supported by Grant No. 12271349 of the NNSFC. The third author was partially supported by grants from the Research Grants Council of the Hong Kong Special Administrative Region, China (Project No. 14300018 and 14301819).
\end{acknowledgements}

\section{Preliminaries} \label{sec: preliminaries}

In this section, we recall some  tools used in this paper. First, Ozsv\'ath and Szab\'o proved that the symplectic fillings of L-spaces must be negative-definite.
\begin{theorem}\cite{OSzGen}\label{thm: negdefinite}
An L-space $Y$ has no symplectic semi-filling with disconnected boundary; and all its symplectic fillings have $b^{+}_{2}(W)=0$. In particular, $Y$ admits no taut foliation.
\end{theorem}

For a rational homology sphere $Y$ with a Spin$^{c}$ structure $\mathfrak{s}$, $d(Y,\mathfrak{s})$ denotes the {\it correction term} or the {\it d-invariant} defined by Ozsv\'ath and Szab\'o in \cite{OSzdinv}. Owens and Strle give a necessary condition for a rational homology sphere to bound a negative-definite 4-manifold using the d-invariants.

\begin{proposition}\cite[Proposition 5.2]{OwSt}\label{proposition: d0}
Let $Y$ be a rational homology sphere with $|H_{1}(Y;\mathbb{Z})|$ $=zw^2$, where both $z$ and $w$ are positive integers, and $z$ is square-free. If $Y$ bounds a negative-definite four-manifold $X$, then $$\max\limits_{\mathfrak{t}\in \text{Spin}^{c}(Y)} 4d(Y,\mathfrak{t})\geq \left\{
                \begin{array}{ll}
                  1-\frac{1}{z},~~~ \text{if}~~ z~~ \text{is odd},\\
                  1, ~~~~~~~~~~~~~~ \text{if}~~ z~~ \text{is even}.

                \end{array}
              \right.$$
\end{proposition}

\medskip
Hence, $Y$ does not bound a negative-definite 4-manifold as long as $\max\limits_{\mathfrak{t}\in \text{Spin}^{c}(Y)} d(Y,\mathfrak{t})< 0.$
Moreover, if $|H_{1}(Y;\mathbb{Z})|$ is not a square, then $Y$ does not bound a negative-definite 4-manifold as long as $\max\limits_{\mathfrak{t}\in \text{Spin}^{c}(Y)} d(Y,\mathfrak{t})< \frac{1}{6}.$
In particular, this together with Theorem \ref{thm: negdefinite} implies that an L-space with small d-invariants admits no symplectic fillable contact structure.  Most of our nonexistence results in this paper will be proved by this observation.  

The following lemma of Owens and Strle is also very useful.

\begin{lemma}\cite[Lemma 2.6]{OwSt1}\label{lemma:negdefcob}
Let $K$ be a knot in $S^3$ and let $r, s$ be rational numbers with $r>s>0$.
Then there exists a negative-definite two-handle cobordism from $S^{3}_{s}(K)$ to $S^{3}_{r}(K)$.
\end{lemma}

An efficient method of constructing an L-space is Dehn surgery along an L-space knot or an L-space link. The d-invariants of such L-spaces depend on the Alexander polynomials or multi-variable Alexander polynomials.



More specifically, one can use the Alexander polynomials to compute the d-invariants of the Dehn surgeries along L-space knots as follows.
\begin{proposition}\cite[Proposition 6.1]{OwSt}\label{proposition: d1}
Let $K\subset S^{3}$ be an L-space knot. Then for any integer $p>0$ the d-invariants of the $p$-surgery on $K$ are given by $$d(S^{3}_{p}(K),i)=d(S^{3}_{p}(U),i)-2t_{i}(K)$$ for $|i|\leq p/2$, where $U$ is the unknot,
\begin{equation} \label{eq:dinvofLp1}
d(S^{3}_{p}(U),i)=\frac{(p-2i)^2}{4p}-\frac{1}{4},
\end{equation}
and
\begin{equation}
\label{t_i}
t_{i}(K)=\sum\limits_{j>0}ja_{|i|+j}
\end{equation}
where the $a_i$ are the coefficients of the symmetrized Alexander polynomial of $K$.
\end{proposition}

Meanwhile, Gorsky, Liu and Moore established a formula for the d-invariants for Dehn-surgeries along two-component L-space links with linking number $0$. We tailor their result to our needs.
\begin{theorem} \cite[Theorem 1.1]{glm}\label{thm: d2}
The $d$-invariants of positive integral surgeries on a two-component L-space link with linking number zero can be computed as follows:  $$d(S^{3}_{p_1,p_2}(\mathbb{L}), (i_1, i_2))=d(S^{3}_{p_1}(U),i_1)+d(S^{3}_{p_2}(U),i_2)-2\max\{h(s_{\pm\pm}(i_1,i_2))\},$$ where $U$ is the unknot, $p_1$ and $p_2$ are positive integers, and $s_{\pm\pm}(i_1,i_2)=(s_{\pm}^{(1)},s_{\pm}^{(2)})$ are four lattice points in Spin$^c$-structure $(i_1,i_2)$ which are closest to the origin in each quadrant.
\end{theorem}

Here, the h-function of $\mathbb{L}$ is determined by the Alexander polynomial of its sublinks as below.

Suppose $$\widetilde{\Delta}(\widehat{\mathbb{L}}):=t_1^{\frac{1}{2}}t_2^{\frac{1}{2}}\Delta(\widehat{\mathbb{L}})
=\sum\limits_{j_{1},j_{2}}a_{j_1,j_2}t_{1}^{j_1}t_{2}^{j_2},$$
where $\Delta(\widehat{\mathbb{L}})$ is the multi-variable Alexander polynomial of the link $\widehat{\mathbb{L}}$, and for $i=1,2$, $$\widetilde{\Delta}(L_i):=\frac{1}{1-t^{-1}}\Delta(L_i)=\sum\limits_{j}a^{i}_{j}t^{j},$$
where $\Delta(L_i)$ is the Alexander polynomial of each of its components $L_i$.

By \cite{gn} and \cite[Theorem 3.7]{bg}, the H-function for $\widehat{\mathbb{L}}$ can be computed by the formula

\begin{equation}\label{DefHfunction}
  H_{\widehat{\mathbb{L}}}(s_1, s_2)
  =\sum\limits_{j\geq s_{1}+1}a^{1}_{j}+\sum\limits_{j\geq s_{2}+1}a^{2}_{j}-\sum\limits_{\substack{j_1\geq s_1+1 \\ j_2\geq s_2+1}}a_{j_1,j_2},
\end{equation}
and the h-function for $\widehat{\mathbb{L}}$ is defined by  $$h(s_1,s_2)=H_{\widehat{\mathbb{L}}}(s_1,s_2)-H_{O}(s_1,s_2),$$ where $O$ is the two-component unlink. 

\bigskip

We also recall a topological characterization of Stein domains given by Eliashberg \cite{e1} and Gompf \cite{g}.

\begin{theorem}\cite[Theorem 1.3]{g}\label{char}
A smooth, compact, connected, oriented 4-manifold $X$ admits a Stein structure (inducing the given orientation) if and only if it can be presented as a handlebody by attaching 2-handles to a framed link in $\partial(D^{4}\cup \text{1-handles})=\#_{m} S^{1}\times S^{2}$, where the link is drawn in a Legendrian standard form and the framing coefficient on each link component $K$ is given by $tb(K)-1$.
\end{theorem}

In practice, we use the following more convenient characterization.

\begin{proposition}\cite[Proposition 5.3]{g}\label{char1}
Let $L$ be a Legendrian link in standard form in $\#_{m} S^{1}\times S^{2}$, with a rational coefficient $r_i$ assigned to each component $K_i$. If $r_{i}<tb(K_i)$ or $r_{i}=\infty$ for each $i$, then the manifold $M$ obtained by rational surgery on $L$ with these coefficients is the oriented boundary of a Stein surface.
\end{proposition}

\section{Non-existence of symplectic fillable contact structures} \label{sec:nonexistence}

\subsection{Dehn surgeries along L-space knots}

According to Vafaee \cite{v}, the twisted torus knot $K_{n,m}$, as depicted in Figure \ref{figure:Knm}, is an L-space knot.

\begin{lemma}\label{alex1}
The Alexander polynomial of $K_{n,m}$ is
\begin{align*}
\Delta(K_{n,m})=&(-1)^{n-1}+\sum_{i=1}^{n-1}(-1)^{n-i-1}(t^i+t^{-i})-\sum_{k=1}^{m}(t^{n+3k-2}+t^{-n-3k+2})\\
&+\sum_{k=1}^m(t^{n+3k-1}+t^{-n-3k+1}). 
\end{align*}


\end{lemma}

\begin{proof}
The closed 3-braid presentation of $K_{n,m}$ is $ \sigma_{1}^{-2n+4}(\sigma_1^{-1}\sigma_{2}^{-1})^{3m+2}$. Here $\sigma_{1}$ and $\sigma_2$ are the generators of the three-stranded braid group, see for example \cite[Fig. 1]{bi}. According to \cite{bi}, one can use the Burau representation to compute the Alexander polynomial. Recall that the Burau representation $\psi:B_3\rightarrow GL(2, \mathbb{Z}[t,t^{-1}])$ is defined as
$$\psi(\sigma^{-1}_{1})=
\begin{bmatrix}
-t & 1 \\
0 & 1
\end{bmatrix},$$
$$\psi(\sigma^{-1}_{2})=
\begin{bmatrix}
1 & 0 \\
t & -t
\end{bmatrix}.$$ Then
\begin{align*}
\psi(\sigma_{1}^{-2n+4}(\sigma_1^{-1}\sigma_{2}^{-1})^{3m+2})&= \begin{bmatrix}
-t & 1 \\
0 & 1
\end{bmatrix}^{2n-4}\begin{bmatrix}
-1 & 1 \\
-1 & 0
\end{bmatrix} t^{3m+2}\\
&=\begin{bmatrix}
-t^{2n-4}+t^{2n-5}-t^{2n-6}+\cdots+t-1 & t^{2n-4} \\
1 & 0
\end{bmatrix}t^{3m+2}.
\end{align*}
By \cite[Eq.(7)]{bi}, the Alexander polynomial $$\Delta(K_{n,m})=\frac{t^{n+3m}-t^{-n-3m}(-t^{2n-4}+t^{2n-5}-t^{2n-6}+\cdots+t-1)t^{3m+2}+t^{-n-3m}}{t^{-1}+1+t}.$$
We then obtain the desired form in the statement by direct computation.
\end{proof}

This polynomial is also obtained in \cite{t}. So the genus of $K_{n,m}$ is $n+3m-1$.

Using (\ref{t_i}), we compute the torsion coefficients of $K_{n,m}$ as follows:
\begin{itemize}

\item If $n$ is odd and  $i\leq n-2$, then $$t_{i}= \left\{
                \begin{array}{ll}

                  m+\frac{n-i}{2}, ~~~~~~~~~~~~~~& \text{if}~~ i~~ \text{is odd},\\
                  m+\frac{n-i-1}{2}, ~~~~~~~~~~~~~~& \text{if}~~ i~~ \text{is even}.

                \end{array}
              \right. $$

\item If $n$ is even and  $i\leq n-2$, then $$t_{i}= \left\{
                \begin{array}{ll}

                  m+\frac{n-i-1}{2}, ~~~~~~~~~~~~~~& \text{if}~~ i~~ \text{is odd},\\
                  m+\frac{n-i}{2}, ~~~~~~~~~~~~~~ & \text{if}~~ i~~ \text{is even}.

                \end{array}
              \right. $$

\item If $i\geq n-1$, then $t_{i}=m-k$ for $i=n+3k+\epsilon$, where $\epsilon\in\{-1,0,1\}$ and $0\leq k\leq m-1$.
\end{itemize}

According to \cite{Rsurgery}, both $(2n+6m-3)$-surgery and  $(2n+6m-2)$-surgery along $K_{n,m}$ are L-spaces.

We consider the $(2n+6m-2)$-surgery $Y:=S^3_{2n+6m-2}(K_{n,m})$.

\begin{lemma}\label{lemma:negd1}
If  $0\leq i\leq n-2$,
then
$d(Y,i)<0.$
\end{lemma}

\begin{proof}
Note that
\begin{align*}
d(Y,i)&=\frac{(2n+6m-2-2i)^2}{4\times (2n+6m-2)}-\frac{1}{4}-2t_{i} \\
&\leq  \frac{(2n+6m-2-2i)^2}{4\times (2n+6m-2)}-\frac{1}{4}-(2m+n-i-1).
\end{align*}
It suffices to show the following inequalities.
\begin{align*}
& \frac{(2n+6m-2-2i)^2}{4\times (2n+6m-2)}<\frac{1}{4}+(2m+n-i-1).\\
\Leftrightarrow \;\;  & \frac{(n+3m-1-i)^2}{2n+6m-2}<2m+n-i-\frac{3}{4}.\\
\Leftrightarrow \;\;  & \frac{(n+3m-1-i)^2}{n+3m-1}<2n+4m-2i-\frac{3}{2}.\\
 \Leftrightarrow \;\;  &    n+3m-1-2i+\frac{i^2}{n+3m-1}<2n+4m-2i-\frac{3}{2}.\\
  \Leftrightarrow \;\;  &    \frac{i^2}{n+3m-1}<n+m-\frac{1}{2}.
\end{align*}
The last equality follows from $0\leq i\leq n-2$.
\end{proof}

\begin{lemma}\label{lemma:negd2}
If  $n-1\leq i \leq n+3m-1$, then $d(Y,i)<0$.
\end{lemma}
\begin{proof}
For $i=n+3m-1$, it is clear that $d(Y,i)=-\frac{1}{4}-2t_{i}<0$ as $t_i$ is non-negative.

Now, we assume $n-1\leq i\leq n+3m-2$. Then we can uniquely express $i=n+3k+\epsilon$, where $\epsilon\in\{-1,0,1\}$ and $0\leq k\leq m-1$; and

\begin{align*}
d(Y,i)&=\frac{(2n+6m-2-2i)^2}{4\times (2n+6m-2)}-\frac{1}{4}-2t_{i}\\
      &= \frac{(2n+6m-2-2(n+3k-1))^2}{4\times (2n+6m-2)}-\frac{1}{4}-2(m-k)
\end{align*}

It suffices to show the following inequalities.

\begin{align*}
  & \frac{(2n+6m-2-2(n+3k-1))^2}{4\times (2n+6m-2)}<\frac{1}{4}+2m-2k.\\
\Leftrightarrow \;\;  & \frac{(3m-3k)^2}{2n+6m-2}<\frac{1}{4}+2m-2k.\\
\Leftrightarrow \;\;  & \frac{9(m-k)^2}{n+3m-1}<\frac{1}{2}+4m-4k.\\
 \Leftarrow \;\;  &    \frac{9(m-k)^2}{n+3m-1}<4m-4k.\\
  \Leftrightarrow \;\;  &     \frac{9(m-k)}{n+3m-1}<4.\\
  \Leftrightarrow \;\;  &      4<3m+4n+9k,
\end{align*}
where we used $0\leq k\leq m-1$ in the second last line of equivalence.
\end{proof}

Similarly, the closed 3-braid presentation of $K'_{n,m}$ is $ \sigma_{1}^{-2n+4}(\sigma_1^{-1}\sigma_{2}^{-1})^{3m+1}$, and one can use the Burau representation to compute its Alexander polynomial.

\begin{lemma}\label{alex2}
The Alexander polynomial of $K'_{n,m}$ is

\begin{align*}
\Delta(K'_{n,m})=&(-1)^n+\sum_{i=1}^{n-2}(-1)^{n-i}(t^i+t^{-i}) -\sum_{k=0}^{m-1}(t^{n+3k}+t^{-n-3k})\\ &+\sum_{k=0}^{m-1}(t^{n +3k+1}+t^{-n-3k-1}).    
\end{align*}

\end{lemma}

The genus of $K'_{n,m}$ is $n+3m-2$. The torsion coefficients of $K'_{n,m}$ is as follows:

\begin{itemize}
\item If $n$ is even and  $i\leq n-3$, then $$t_{i}= \left\{
                \begin{array}{ll}

                  m+\frac{n-i-2}{2}, ~~~~~~~~~~~~~~ \text{if}~~ i~~ \text{is even},\\
                  m+\frac{n-i-1}{2}, ~~~~~~~~~~~~~~ \text{if}~~ i~~ \text{is odd}.

                \end{array}
              \right. $$

\item If $n$ is odd and  $i\leq n-3$, then $$t_{i}= \left\{
                \begin{array}{ll}

                  m+\frac{n-i-1}{2}, ~~~~~~~~~~~~~~ \text{if}~~ i~~ \text{is even},\\
                  m+\frac{n-i-2}{2}, ~~~~~~~~~~~~~~ \text{if}~~ i~~ \text{is odd}.

                \end{array}
              \right. $$


\item If $i\geq n-2$, then $t_{i}=m-k$ for $i=n+3k-1+\epsilon$, where $\epsilon\in\{-1,0,1\}$ and $0\leq k\leq m-1$.

\end{itemize}

According to \cite{Rsurgery}, both $(2n+6m-5)$-surgery and $(2n+6m-4)$-surgery along $K'_{n,m}$ are L-spaces.

We consider the $(2n+6m-4)$-surgery  $Y':=S^3_{2n+6m-4}(K'_{n,m})$.

\begin{lemma}\label{lemma:negd3}
If  $0\leq i\leq n-3$, then
$d(Y',i)<0$.
\end{lemma}

\begin{proof}
Note that
\begin{align*}
d(Y',i)&=\frac{(2n+6m-4-2i)^2}{4\times (2n+6m-4)}-\frac{1}{4}-2t_{i}\\
&\leq \frac{(2n+6m-4-2i)^2}{4\times (2n+6m-4)}-\frac{1}{4}-(2m+n-i-2).
\end{align*}
It suffices to show the following inequalities.
\begin{align*}
& \frac{(2n+6m-4-2i)^2}{4\times (2n+6m-4)}<\frac{1}{4}+(2m+n-i-2).\\
\Leftrightarrow \;\;  & \frac{(n+3m-2-i)^2}{2n+6m-4}<2m+n-i-\frac{7}{4}. \\
\Leftrightarrow \;\;  & \frac{(n+3m-2-i)^2}{n+3m-2}<2n+4m-2i-\frac{7}{2}. \\
\Leftrightarrow \;\;  & n+3m-2-2i+\frac{i^2}{n+3m-2}<2n+4m-2i-\frac{7}{2}. \\
\Leftrightarrow \;\;  & \frac{i^2}{n+3m-1}<n+m-\frac{3}{2}.
\end{align*}
The last inequality follows from $0\leq i\leq n-3$.
\end{proof}

\begin{lemma}\label{lemma:negd4}
If $n-2\leq i\leq n+3m-2$, then $d(Y',i)<0$.
\end{lemma}

\begin{proof}

For $i=n+3m-2$, it is clear that $d(Y',i)=-\frac{1}{4}-2t_{i}<0$ as $t_i$ is non-negative.

Now, we assume $n-2\leq i\leq n+3m-3$; then we can uniquely express $i=n+3k-1+\epsilon$, where $\epsilon\in\{-1,0,1\}$ and $0\leq k\leq m-1$; and

\begin{align*}
d(Y',i)&=\frac{(2n+6m-4-2i)^2}{4\times (2n+6m-4)}-\frac{1}{4}-2t_{i} \\
       & =\frac{(2n+6m-4-2i)^2}{4\times (2n+6m-4)}-\frac{1}{4}-2(m-k).
\end{align*}

It suffices to show the following inequalities.
\begin{align*}
& \frac{(2n+6m-4-2(n+3k-2))^2}{4\times (2n+6m-4)}<\frac{1}{4}+2m-2k. \\
\Leftrightarrow \;\;  &  \frac{(3m-3k)^2}{2n+6m-4}<\frac{1}{4}+2m-2k. \\
\Leftrightarrow \;\;  &  \frac{9(m-k)^2}{n+3m-2}<\frac{1}{2}+4m-4k. \\
\Leftarrow \;\;  &  \frac{9(m-k)^2}{n+3m-2}<4m-4k. \\
\Leftrightarrow \;\;  &  \frac{9(m-k)}{n+3m-2}<4.  \\
\Leftrightarrow \;\;  &  9(m-k)<4n+12m-8. \\
\Leftrightarrow \;\;  &  8<4n+3m+9k.
\end{align*}
The last inequality follows from $n\geq2$ and $m\geq 1$.
\end{proof}

\begin{proof}[Proof of Theorem~\ref{Theorem:Main1}]
First, we prove  the case for $K_{n,m}$. By Lemma~\ref{lemma:negd1}, Lemma~\ref{lemma:negd2} and Proposition~\ref{proposition: d0}, the 3-manifold $S^{3}_{2n+6m-2}(K_{n,m})$ does not bound a negative-definite 4-manifold. Furthermore, Lemma~\ref{lemma:negdefcob} implies that $S^{3}_{r}(K_{n,m})$ does not bound a negative-definite 4-manifold either for any $r\in[2n+6m-3, 2n+6m-2]$. Since $S^{3}_{r}(K_{n,m})$ is an L-space for $r\geq 2g(K_{n,m})-1=2n+6m-3$, the theorem follows from  Theorem~\ref{thm: negdefinite}.

The proof for the case of $K'_{n,m}$ is the same if we apply Lemma~\ref{lemma:negd3} and Lemma~\ref{lemma:negd4} instead.
\end{proof}


\begin{proof}[Proof of Theorem~\ref{Theorem:Main2}]
Note that the $(2g-1)$-surgery along the L-space knot $K$ is an L-space, which we denote by $Y:=S^3_{2g-1}(K)$.  By Theorem~\ref{thm: negdefinite}, Proposition~\ref{proposition: d0} and Proposition~\ref{proposition: d1}, it suffices to show that $d(Y,i)<0$ for $i=0,1,\cdots,g-1$, where $$d(Y,i)=\frac{(2g-1-2i)^2}{4(2g-1)}-\frac{1}{4}-2t_{i}(K).$$


Recall that $i_{k}=\min\{i\in\mathbb{Z}_{\geq 0} \mid 2i>2g-1-\sqrt{(8k+1)(2g-1)}\}$ for integer $k\in [0, [\frac{g-1}{4}]+1]$. We obtain a sequence of integers
$$0=i_{[\frac{g-1}{4}]+1}\leq i_{[\frac{g-1}{4}]}\leq \cdots \leq i_1 \leq i_0 \leq g-1.$$
Note that if $i\geq i_{k}$, then $2i>2g-1-\sqrt{(8k+1)(2g-1)}$, which implies $\frac{(2g-1-2i)^2}{4(2g-1)}-\frac{1}{4}<2k$.  It follows that $$d(Y,i)<2k-2t_i(K).$$

From this observation, we prove $d(Y,i)<0$ as follows:

Suppose $i_0 \leq i \leq g-1$.  It readily follows that $d(Y,i)<-2t_i(K)\leq 0$.

Suppose $i_k\leq i \leq i_{k-1}$ for some $k\in \{1, 2, \cdots, [\frac{g-1}{4}]+1\}$.   By the monotonicity of $t_i$, we have $t_i(K)\geq t_{i_{k-1}}(K)\geq k$, where we use the assumption on the torsion coefficient for the second inequality.  Hence $d(Y,i)<2k-2t_i(K)\leq 0$.
\end{proof}

\subsection{Dehn surgeries along two-component L-space links} \label{section: links}

In this section, we consider Dehn surgeries along the two-component L-space links $\mathbb{L}_{n}$ and $K(5,5)$.  Our strategy for proving Theorem ~\ref{Theorem:Main3} and Proposition~\ref{proposition:K(5,5)} is similar to that for the case of surgeries on L-space knots in the previous subsection. On one hand, we prove the surgered manifolds are L-spaces; on the other hand, we show that in some cases their $d$-invariants are negative or smaller than $\frac{1}{6}$.  The non-existence of a symplectic fillable contact structure then follows from Theorem~\ref{thm: negdefinite} and Proposition~\ref{proposition: d0}.


\begin{lemma}\label{lemma:-defcob}
Let $K_{1}\cup K_{2}\subset S^3$ be a two-component link with linking number $0$, and let $r_{i}, s_{i}$ be rational numbers with $r_{i} \geq s_{i}>0$, for $i=1,2$.
Then there exists a negative-definite two-handle cobordism from $S_{s_1,s_2}^{3}(K_{1}\cup K_{2})$ to $S_{r_1,r_2}^{3}(K_{1}\cup K_{2})$.
\end{lemma}

\begin{proof}
Since the linking number of $K_{1}\cup K_{2}$ is $0$, the proof is essentially the same as the proof of Lemma~\ref{lemma:negdefcob} in \cite{OwSt1}.
\end{proof}

\begin{lemma}\label{lemma:lspace}
The $(p_1, p_2)$-surgery along $\mathbb{L}_{n}$ is an L-space if $p_1\geq 1$ and $p_2\geq 2n+1$.
\end{lemma}

\begin{proof}
By blowing down, the $(1, 2n+1)$-surgery along  $\mathbb{L}_{n}$ is equivalent to an $(2n+1)$-surgery along the torus knot $T(2, 2n+3)$ which is an L-space. The $(2n+1)$-surgery along the torus knot $T(2,2n+1)$ is an L-space. So by the L-space surgery induction lemma \cite[Lemma 2.5]{liu}, $S^{3}_{p_1, 2n+1}(\mathbb{L}_{n})$ is an L-space for all positive integer $p_1$.  Now as the $p_1$-surgery along the unknot is an L-space, we can apply the L-space surgery induction lemma again and conclude that $S^{3}_{p_1, p_2}(\mathbb{L}_{n})$ is an L-space for all integers $p_1\geq 1$ and $p_2\geq 2n+1$.
\end{proof}

\begin{lemma}\label{lemma:alex3}
The two-variable Alexander polynomial of $\mathbb{L}_{n}$
$$\Delta(\mathbb{L}_{n})=-(t_{1}^{\frac{1}{2}}-t_{1}^{-\frac{1}{2}})(t_{2}^{\frac{2n+1}{2}}-t_{2}^{\frac{2n-1}{2}}+ \cdots +t_{2}^{-\frac{2n-1}{2}}-t_{2}^{-\frac{2n+1}{2}}).$$
\end{lemma}

\begin{proof}
We denote the Conway potential function of $\mathbb{L}_{n}$ by $\nabla(\mathbb{L}_{n})$. See \cite{Ha}. It is related to the two-variable Alexander polynomial in the following way, $$\nabla(\mathbb{L}_{n})(t_1,t_2)=\Delta(\mathbb{L}_{n})(t^2_1,t^2_2).$$

We observe that $\mathbb{L}_{0}$ is the Whitehead link and $\mathbb{L}_{1}$ is $L7a3$, and so $$\nabla(\mathbb{L}_{0})(t_1,t_2)=-(t_{1}-t_{1}^{-1})(t_{2}-t_{2}^{-1}),$$  $$\nabla(\mathbb{L}_{1})(t_1,t_2)=-(t_{1}-t_{1}^{-1})(t^{3}_{2}-t_{2}+t_{2}^{-1}-t_{2}^{-3}).$$ Assume now that $$\nabla(\mathbb{L}_{k})(t_1,t_2)=-(t_{1}-t_{1}^{-1})(t_{2}^{2k+1}-t_{2}^{2k-1}+ \cdots +t_{2}^{-2k+1}-t_{2}^{-2k-1}),$$ for $k\leq n$. By the main result in \cite{j},
\begin{align*}
\nabla(\mathbb{L}_{n})(t_1,t_2)&=(t_{2}^{2}+t_{2}^{-2})\nabla(\mathbb{L}_{n-1})(t_1,t_2)-\nabla(\mathbb{L}_{n-2})(t_1,t_2)\\
&=-(t_{1}-t_{1}^{-1})(t_{2}^{2n+1}-t_{2}^{2n-1}+ \cdots +t_{2}^{-2n+1}-t_{2}^{-2n-1}),
\end{align*}
which completes the induction. The last step is to obtain the desired from of the two-variable Alexander polynomial of $\mathbb{L}_{n}$ from the straightforward change of variables.
\end{proof}

Thus, we have
$$\widetilde{\Delta}(\mathbb{L}_{n})=t_{1}^{\frac{1}{2}}t_{2}^{\frac{1}{2}}\Delta(\mathbb{L}_{n})=-(t_{1}-1)(t_{2}^{n+1}-t_{2}^{n}+ \cdots +t_{2}^{-n+1}-t_{2}^{-n}).$$

The component $L_1$ is the unknot, so the Alexander polynomial of $L_1$ is
$\Delta(L_1)=1$, and $$\widetilde{\Delta}(L_1)=\frac{1}{1-t^{-1}}\Delta(L_1)=1+t^{-1}+t^{-2}+\cdots.$$

The component $L_2$ is the torus knot $T(2,2n+1)$. The Alexander polynomial of $L_2$ is
$$\Delta(L_2)=t^{n}-t^{n-1}+\cdots+t^{-n+1}-t^{-n},$$ and
$$\widetilde{\Delta}(L_2)=\frac{1}{1-t^{-1}}\Delta(L_2)=t^{n}+t^{n-2}+\cdots+t^{-n+2}+t^{-n}+t^{-n-1}+t^{-n-2}+\cdots.$$

Using (\ref{DefHfunction}), we compute the $h$-function of $\mathbb{L}_n$ as follows.

\begin{itemize}
    \item If $n$ is even, and $s_1=0$, then $$h(0,k)= \left\{
                \begin{array}{ll}

                  \max\{\frac{n-|k|+1}{2},0\}, ~~~~~~~~~~~~~~ \text{if}~~ |k|~~ \text{is odd},\\
                  \max\{\frac{n-|k|+2}{2},0\}, ~~~~~~~~~~~~~~ \text{if}~~ |k|~~ \text{is even}.

                \end{array}
              \right. $$

   \item If $n$ is even, and $s_1\neq 0$, then            $$h(s_1,k)= \left\{
                \begin{array}{ll}

                  \max\{\frac{n-|k|+1}{2},0\}, ~~~~~~~~~~~~~~ \text{if}~~ |k|~~ \text{is odd},\\
                  \max\{\frac{n-|k|}{2},0\}, ~~~~~~~~~~~~~~ \text{if}~~ |k|~~ \text{is even}.

                \end{array}
              \right. $$

   \item If $n$ is odd, and $s_1=0$, then $$h(0,k)= \left\{
                \begin{array}{ll}

                  \max\{\frac{n-|k|+2}{2},0\}, ~~~~~~~~~~~~~~ \text{if}~~ |k|~~ \text{is odd},\\
                  \max\{\frac{n-|k|+1}{2},0\}, ~~~~~~~~~~~~~~ \text{if}~~ |k|~~ \text{is even}.

                \end{array}
              \right. $$

   \item If $n$ is odd, and $s_1\neq 0$, then            $$h(s_1,k)= \left\{
                \begin{array}{ll}

                  \max\{\frac{n-|k|}{2},0\}, ~~~~~~~~~~~~~~ \text{if}~~ |k|~~ \text{is odd},\\
                  \max\{\frac{n-|k|+1}{2},0\}, ~~~~~~~~~~~~~~ \text{if}~~ |k|~~ \text{is even}.

                \end{array}
              \right. $$
    \end{itemize}

We employ Theorem~\ref{thm: d2} to compute the d-invariants of $d(S^{3}_{p_1, p_2}(\mathbb{L}_{n}), (i_1, i_2))$, where $0\leq i_1\leq p_{1}-1$ and $0\leq i_2\leq p_{2}-1$, and $p_1$ and $p_2$ are positive integers.

Suppose $(p_1,p_2)=(4,2n+1)$, or $(p_1,p_2)=(5,2n+1)$ and $5<2n+1$. We first observe the following symmetry: For $n+1\leq k'\leq 2n$, we have
$$\frac{(2n+1-2k')^2}{4(2n+1)}-\frac{1}{4}=\frac{(2n+1-2k)^2}{4(2n+1)}-\frac{1}{4},$$ where $1\leq k:=2n+1-k'\leq n$. 
Also note that  $(i_1,k')$ and $(i_1,-k)$ belong to the same Spin$^c$ structure of $S^{3}_{p_{1}, 2n+1}(\mathbb{L}_{n})$, and $h(i_1,-k)=h(i_1,k)\geq h(i_1, k')$.
Thus, in order to prove that $d(S^{3}_{p_{1}, 2n+1}(\mathbb{L}_{n}), (i_1, i_2))$ is negative for $0\leq i_1\leq p_{1}-1$ and $0\leq i_2\leq 2n$, it suffices to show that
$$\frac{(p_{1}-2i_1)^2}{4p_1}-\frac{1}{4}+\frac{(2n+1-2k)^2}{4(2n+1)}-\frac{1}{4}<2h(i_{1},k)$$  for $0\leq k\leq n$.

\begin{lemma}\label{lemma:subcase1}
(1.1) If $(i_1, i_2)=(0,k)$, where $0\leq k\leq n$, $n$ is even and $k$ is odd. Then $h(0,k)=\frac{n-k+1}{2}$, and
$$\frac{p_{1}-1}{4}+\frac{(2n+1-2k)^2}{4(2n+1)}-\frac{1}{4}<n-k+1.$$

(1.2) If $(i_1, i_2)=(i_1,k)$, where $i_1\neq0$, $0\leq k\leq n$, $n$ is even and $k$ is odd. Then $h(i_1,k)=\frac{n-k+1}{2}$, and
$$\frac{(p_{1}-2i_1)^2}{4p_1}-\frac{1}{4}+\frac{(2n+1-2k)^2}{4(2n+1)}-\frac{1}{4}<n-k+1.$$

(2.1) If $(i_1, i_2)=(0,k)$, where $0\leq k\leq n$, $n$ is even and $k$ is even. Then $h(0,k)=\frac{n-k+2}{2}$, and
$$\frac{p_{1}-1}{4}+\frac{(2n+1-2k)^2}{4(2n+1)}-\frac{1}{4}<n-k+2.$$

(2.2) If $(i_1, i_2)=(i_1,k)$, where $0<i_1\leq\frac{p_1}{2}$, $0\leq k\leq n$, $n$ is even and $k$ is even. Then $h(i_1,k)=\frac{n-k}{2}$, and
$$\frac{(p_{1}-2i_1)^2}{4p_1}-\frac{1}{4}+\frac{(2n+1-2k)^2}{4(2n+1)}-\frac{1}{4}<n-k.$$

(3.1) If $(i_1, i_2)=(0,k)$, where $0\leq k\leq n$, $n$ is odd and $k$ is even. Then $h(0,k)=\frac{n-k+1}{2}$, and
$$\frac{p_{1}-1}{4}+\frac{(2n+1-2k)^2}{4(2n+1)}-\frac{1}{4}<n-k+1.$$

(3.2) If $(i_1, i_2)=(i_1,k)$, where $i_1\neq0$, $0\leq k\leq n$, $n$ is odd and $k$ is even. Then $h(i_1,k)=\frac{n-k+1}{2}$, and
$$\frac{(p_{1}-2i_1)^2}{4p_1}-\frac{1}{4}+\frac{(2n+1-2k)^2}{4(2n+1)}-\frac{1}{4}<n-k+1.$$

(4.1) If $(i_1, i_2)=(0,k)$, where $0\leq k\leq n$, $n$ is odd and $k$ is odd. Then $h(0,k)=\frac{n-k+2}{2}$, and
$$\frac{p_{1}-1}{4}+\frac{(2n+1-2k)^2}{4(2n+1)}-\frac{1}{4}<n-k+2.$$

(4.2) If $(i_1, i_2)=(i_1,k)$, where $i_1\neq0$, $0\leq k\leq n$, $n$ is odd and $k$ is odd. Then $h(i_1,k)=\frac{n-k}{2}$, and
$$\frac{(p_{1}-2i_1)^2}{4p_1}-\frac{1}{4}+\frac{(2n+1-2k)^2}{4(2n+1)}-\frac{1}{4}<n-k.$$
\end{lemma}

\begin{proof}
We prove the cases (1.1) and (2.2), and all other cases are similar.

Proof for the case (1.1):

\begin{align*}
&\frac{(2n+1-2k)^2}{2n+1}<4n-4k+1.\\
\Leftrightarrow \;\;  &  2n+1-4k+\frac{4k^2}{2n+1}<4n-4k+1.\\
\Leftrightarrow \;\;  & \frac{4k^2}{2n+1}<2n.
\end{align*}
The last inequality follows from $0\leq k\leq n$.

\bigskip

Proof for the case (2.2):  Suppose $p_{1}=4$, then $\frac{(p_{1}-2i_1)^2}{4p_1}-\frac{1}{4}\leq 0$ for $i_1\neq 0$. It suffices to show that

\begin{align*}
&\frac{(2n+1-2k)^2}{4(2n+1)}-\frac{1}{4}<n-k.\\
\Leftrightarrow \;\;  & \frac{(2n+1-2k)^2}{2n+1}<4n-4k+1.\\
\Leftrightarrow \;\;  & 2n+1-4k+\frac{4k^2}{2n+1}<4n-4k+1.\\
\Leftrightarrow \;\;  & \frac{4k^2}{2n+1}<2n.
\end{align*}
The last inequality follows from $0\leq k\leq n$.


Suppose $p_{1}=5$, then $\frac{(p_{1}-2i_1)^2}{4p_1}-\frac{1}{4}\leq \frac{1}{5}$ for $i_1\neq 0$. It suffices to show that

\begin{align*}
& \frac{1}{5}+\frac{(2n+1-2k)^2}{4(2n+1)}-\frac{1}{4}<n-k.\\
\Leftrightarrow \;\;  & \frac{(2n+1-2k)^2}{2n+1}<4n-4k+\frac{1}{5}.\\
\Leftrightarrow \;\;  & 2n+1-4k+\frac{4k^2}{2n+1}<4n-4k+\frac{1}{5}.\\
\Leftrightarrow \;\;  & \frac{4k^2}{2n+1}<2n-\frac{4}{5}.\\
\Leftrightarrow \;\;  & 4k^2<4n^{2}+\frac{2}{5}n-\frac{4}{5}.
\end{align*}

The last inequality follows from $0\leq k\leq n$ and $5<2n+1$.
\end{proof}


Suppose $(p_{1}, p_{2})=(2, 2n+2)$ and $n+1$ is not a square. Similar to the previous case, 
in order to prove that  $d(S^{3}_{2, 2n+2}(\mathbb{L}_{n}), (i_1, i_2))<\frac{1}{6}$ for $0\leq i_1\leq 1$ and $0\leq i_2\leq 2n+1$, it suffices to show that
$$\frac{(p_1-2i_1)^2}{4p_1}-\frac{1}{4}+\frac{(2n+2-2k)^2}{4(2n+2)}-\frac{1}{4}-2h(i_{1},k)<\frac{1}{6}$$ for $0\leq k\leq n+1$.

\begin{lemma}\label{lemma:subcase2}
(1.1) If $(i_1, i_2)=(0,k)$, where $0\leq k\leq n$, $n$ is even and $k$ is odd. Then $h(0,k)=\frac{n-k+1}{2}$, and
$$\frac{1}{4}+\frac{(2n+2-2k)^2}{4(2n+2)}-\frac{1}{4}<n-k+1.$$

(1.2) If $(i_1, i_2)=(0,n+1)$. Then $h(0,n+1)=0$, and $$\frac{1}{4}+\frac{(2n+2-2k)^2}{4(2n+2)}-\frac{1}{4}-2h(0, n+1)=0<\frac{1}{6}.$$

(1.3) If $(i_1, i_2)=(i_1,k)$, where $i_1\neq0$, $0\leq k\leq n$, $n$ is even and $k$ is odd. Then $h(i_1,k)=\frac{n-k+1}{2}$, and
$$-\frac{1}{4}+\frac{(2n+2-2k)^2}{4(2n+2)}-\frac{1}{4}<n-k+1.$$

(1.4) If $(i_1, i_2)=(i_1,n+1)$. Then $h(i_1,n+1)=0$, $$-\frac{1}{4}+\frac{(2n+2-2k)^2}{4(2n+2)}-\frac{1}{4}-2h(i_{1}, n+1)<0.$$

(2.1) If $(i_1, i_2)=(0,k)$, where $0\leq k\leq n$, $n$ is even and $k$ is even. Then $h(0,k)=\frac{n-k+2}{2}$, and
$$\frac{1}{4}+\frac{(2n+2-2k)^2}{4(2n+2)}-\frac{1}{4}<n-k+2.$$

(2.2) If $(i_1, i_2)=(i_1,k)$, where $i_1\neq0$, $0\leq k\leq n$, $n$ is even and $k$ is even. Then $h(i_1,k)=\frac{n-k}{2}$, and
$$-\frac{1}{4}+\frac{(2n+2-2k)^2}{4(2n+2)}-\frac{1}{4}<n-k.$$

(3.1) If $(i_1, i_2)=(0,k)$, where $0\leq k\leq n$, $n$ is odd and $k$ is even. Then $h(0,k)=\frac{n-k+1}{2}$, and
$$\frac{1}{4}+\frac{(2n+2-2k)^2}{4(2n+2)}-\frac{1}{4}<n-k+1.$$
If $(i_1, i_2)=(0,n+1)$. Then $h(0,n+1)=0$, and $$\frac{1}{4}+\frac{(2n+2-2k)^2}{4(2n+2)}-\frac{1}{4}-2h(0, n+1)=0<\frac{1}{6}.$$

(3.2) If $(i_1, i_2)=(i_1,k)$, where $i_1\neq0$, $0\leq k\leq n$, $n$ is odd and $k$ is even. Then $h(i_1,k)=\frac{n-k+1}{2}$, and
$$-\frac{1}{4}+\frac{(2n+2-2k)^2}{4(2n+2)}-\frac{1}{4}<n-k+1.$$
If $(i_1, i_2)=(i_1,n+1)$. Then $h(i_1,n+1)=0$, and $$-\frac{1}{4}+\frac{(2n+2-2k)^2}{4(2n+2)}-\frac{1}{4}-2h(i_{1}, n+1)<0.$$

(4.1) If $(i_1, i_2)=(0,k)$, where $0\leq k\leq n$, $n$ is odd and $k$ is odd. Then $h(0,k)=\frac{n-k+2}{2}$, and
$$\frac{1}{4}+\frac{(2n+2-2k)^2}{4(2n+2)}-\frac{1}{4}<n-k+2.$$

(4.2) If $(i_1, i_2)=(i_1,k)$, where $i_1\neq0$, $0\leq k\leq n$, $n$ is odd and $k$ is odd. Then $h(i_1,k)=\frac{n-k}{2}$, and
   $$-\frac{1}{4}+\frac{(2n+2-2k)^2}{4(2n+2)}-\frac{1}{4}<n-k.$$
\end{lemma}

\begin{proof} We prove the cases (1.1) and (2.2), and all other cases are similar.

Proof of case (1.1): The inequality is equivalent to
\begin{align*}
    &\frac{(2n+2-2k)^2}{4(2n+2)}<n-k+1\\
  \Leftrightarrow \;\;  &   \frac{n+1-k}{2n+2}<1,
\end{align*}
which holds as $0\leq k \leq n$.

Proof of case (2.2): The inequality if equivalent to
\begin{align*}
   &\frac{(n+1-k)^2}{2n+2}<n-k+\frac{1}{2}\\
     \Leftrightarrow \;\;  &      n+1-2k +\frac{k^2}{n+1}<2n-2k+1\\
     \Leftrightarrow \;\;  &\frac{k^2}{n+1}<n,
\end{align*}
which holds as $0\leq k \leq n$.
\end{proof}




\begin{proof}[Proof of Theorem~\ref{Theorem:Main3}]
Suppose $M=S^{3}_{p_1,p_2}(\mathbb{L}_n)$ where $p_1,p_2$ satisfies the given assumptions. It follows from
Lemma~\ref{lemma:subcase1}, Lemma~\ref{lemma:subcase2}, Proposition~\ref{proposition: d0} and Lemma~\ref{lemma:-defcob} that $M$ does not bound a negative-definite 4-manifold.
On the other hand, by Lemma~\ref{lemma:lspace}, $M$ is an L-space. So Theorem~\ref{thm: negdefinite} implies that $M$ admits no symplectic fillable contact structure.
\end{proof}

\begin{remark}
In fact, by Lemma~\ref{lemma:-defcob} and the proof of Theorem~\ref{Theorem:Main3}, if $(r_1, r_2)$-surgery along $\mathbb{L}_n$ yields an L-space, where the rationals $r_1\in[1,2]$, $r_2\in[2n+1,2n+2]$ and $n+1$ is not a square, or $r_1\in[1,4]$ and $r_{2}=2n+1$, then $S^{3}_{r_1, r_2}(\mathbb{L}_n)$ admits no symplectic fillable contact structure.
\end{remark}

\begin{proof}[Proof of Proposition~\ref{proposition:K(5,5)}]

By \cite[Theorem 3.8]{liu}, the two-bridge link $K(a_{1}, a_{2})$ with both $a_{1}$ and $a_{2}$ positive odd integers are L-space links.

By \cite[Example 5.4]{glm}, the $(3,3)$-surgery along $K(5, 5)$ results in an L-space.   So, according to Theorem~\ref{thm: negdefinite}, any symplectic filling of a contact $S^{3}_{3,3}(K(5, 5))$ is negative-definite.

The Alexander polynomial of $K(5, 5)$ is
$$\Delta(K(5, 5))=-t_{1}^{-\frac{1}{2}}t_{2}^{-\frac{3}{2}}+t_{1}^{\frac{1}{2}}t_{2}^{-\frac{3}{2}}-t_{1}^{-\frac{3}{2}}t_{2}^{-\frac{1}{2}}
+t_{1}^{-\frac{1}{2}}t_{2}^{-\frac{1}{2}}-t_{1}^{\frac{1}{2}}t_{2}^{-\frac{1}{2}}+t_{1}^{\frac{3}{2}}t_{2}^{-\frac{1}{2}}$$$$
+t_{1}^{-\frac{3}{2}}t_{2}^{\frac{1}{2}}-t_{1}^{-\frac{1}{2}}t_{2}^{\frac{1}{2}}+t_{1}^{\frac{1}{2}}t_{2}^{\frac{1}{2}}
-t_{1}^{\frac{3}{2}}t_{2}^{\frac{1}{2}}+t_{1}^{-\frac{1}{2}}t_{2}^{\frac{3}{2}}-t_{1}^{\frac{1}{2}}t_{2}^{\frac{3}{2}}.$$

So $$\widetilde{\Delta}(K(5, 5))=-t_{2}^{-1}+t_{1}t_{2}^{-1}-t_{1}^{-1}+1-t_{1}+t_{1}^{2}+t_{1}^{-1}t_{2}-t_{2}+t_{1}t_{2}-t_{1}^{2}t_{2}+t_{2}^{2}
-t_{1}t_{2}^{2}.$$

By direct computation, we find $h(s_1,s_2)=1$ for $(s_1,s_2)=(0,0),$ $(0,1),(0,-1),(1,0)$ or $(-1,0)$; and $h(s_1,s_2)=0$, otherwise.

Now we compute the d-invariants of $d(S^{3}_{3,3}(K(5, 5)),(i_1,i_2))$ for $0\leq i_1, i_2\leq 2$ using Theorem~\ref{thm: d2}.
$$d(S^{3}_{3,3}(K(5, 5)),(0,0))=\frac{1}{2}+\frac{1}{2}-2h(0,0)=-1,$$
$$d(S^{3}_{3,3}(K(5, 5)),(0,1))=\frac{1}{2}+\frac{-1}{6}-2h(0,1)=-\frac{5}{3},$$
$$d(S^{3}_{3,3}(K(5, 5)),(0,2))=\frac{1}{2}+\frac{-1}{6}-2h(0,-1)=-\frac{5}{3},$$
$$d(S^{3}_{3,3}(K(5, 5)),(1,0))=\frac{-1}{6}+\frac{1}{2}-2h(1,0)=-\frac{5}{3},$$
$$d(S^{3}_{3,3}(K(5, 5)),(1,1))=\frac{-1}{6}+\frac{-1}{6}-2h(1,1)=-\frac{1}{3},$$
$$d(S^{3}_{3,3}(K(5, 5)),(1,2))=\frac{-1}{6}+\frac{-1}{6}-2h(1,2)=-\frac{1}{3},$$
$$d(S^{3}_{3,3}(K(5, 5)),(2,0))=\frac{-1}{6}+\frac{1}{2}-2h(-1,0)=-\frac{5}{3},$$
$$d(S^{3}_{3,3}(K(5, 5)),(2,1))=\frac{-1}{6}+\frac{-1}{6}-2h(2,1)=-\frac{1}{3},$$
$$d(S^{3}_{3,3}(K(5, 5)),(2,2))=\frac{-1}{6}+\frac{-1}{6}-2h(2,2)=-\frac{1}{3}.$$

Since all the d-invariants are negative, by Proposition~\ref{proposition: d0}, $S^{3}_{3,3}(K(5, 5))$ does not bound a negative-definite 4-manifold.

Therefore, $S^{3}_{3,3}(K(5, 5))$ admits no symplectic fillable contact structure.
\end{proof}

\subsection{Non-L-space}

\begin{proof}[Proof of Proposition~\ref{proposition:Main5}]
Suppose $(W,\omega)$ is a Stein filling of $(Y,\xi)$, where the underlying Spin$^c$ structure of $\xi$ is $\mathfrak{s}$. By \cite[Theorem 1]{lin}, if $b^{+}_{2}(W)>0$, then $\mathfrak{s}_{\omega}$ is self-conjugate.
Moreover, we have either of the following two cases:

(Type I)\;\; $b^{+}_{2}(W)=1$, $b^{-}_{2}(W)=4d(Y,\mathfrak{s})+9$, and $HF^{+}(Y)=\mathcal{T}^{+}_{d}\oplus \mathbb{F}_{d}$,

(Type II)\;\; $b^{+}_{2}(W)=2$, $b^{-}_{2}(W)=4d(Y,\mathfrak{s})+10$, and $HF^{+}(Y)=\mathcal{T}^{+}_{d}\oplus \mathbb{F}_{d-1}$.


Thus, if $d(Y,\mathfrak{s})<-2.5$, then $b^{+}_{2}(W)=0$; if $d(Y,\mathfrak{s})<-2.25$ and $HF^{+}(Y)=\mathcal{T}^{+}_{d}\oplus \mathbb{F}_{d}$, then $b^{+}_{2}(W)=0$. However, since all correction terms of $Y$ are negative, Proposition~\ref{proposition: d0} implies  $b^{+}_{2}(W)>0$. We arrive at a contradiction, and this proves the first two cases of the proposition.

On the other hand, we can apply \cite[Theorem 7]{lin1} to the spin cobordism $W \setminus B^4$ between $S^3$ and $Y$.  Note that $\alpha(S^3)=\beta(S^3)=\gamma(S^3)=0$.  If $b^{+}_{2}(W)=1$, then $\alpha(Y,\mathfrak{s})\geq \frac{1}{8}(b_2^-(W)-1)=\frac{d(Y,\mathfrak{s})}{2}+1$ since $b^{-}_{2}(W)=4d(Y,\mathfrak{s})+9$ for (Type I) manifolds; if $b^{+}_{2}(W)=2$, then $\alpha(Y,\mathfrak{s})\geq \frac{1}{8}(b_2^-(W)-2)=\frac{d(Y,\mathfrak{s})}{2}+1$ since  $b^{-}_{2}(W)=4d(Y,\mathfrak{s})+10$ for (Type II) manifolds. So we cannot have $b^{+}_{2}(W)>0$ when $\alpha(Y,\mathfrak{s})< \frac{d(Y,\mathfrak{s})}{2}+1$. The only remaining possibility is $b^{+}_{2}(W)=0$.  When all correction terms of $Y$ are negative, this again contradicts to Proposition~\ref{proposition: d0} and proves the last case of the proposition.
\end{proof}


\begin{example}
Let $T$ be a torus with one boundary component, and $$\phi=(xy)^{9}xy^{-a_1}xy^{-a_2}\cdots xy^{-a_n}$$  a self-diffeomorphism of $T$ which fixes the boundary $\partial T$, where $x,y$ are two right handed Dehn twists along a pair of two simple closed curves on $T$ which transversely intersects in a point,  $a_{i}$ is a nonnegative integer and some $a_{i}$ is positive. Let $(T,\phi)$ be the open book decomposition which supports a contact 3-manifold $(M_{T,\phi}, \xi)$.  Let $\mathfrak{s}_\xi$ be the Spin$^c$ structure of $M_{T,\phi}$ to which the contact structure $\xi$ belongs, then $\mathfrak{s}_\xi$ is self-conjugate \cite[Lemma 6.1]{eo}.  By \cite[Theorem 6.2]{b}, $$HF^{+}(M_{T,\phi}, \mathfrak{s}_\xi)\cong HF^{+}(S^{3}_{-1}(T(2,3))\{d\},$$ where $HF^{+}(S^{3}_{-1}(T(2,3))\cong\mathcal{T}^{+}_{0}\oplus \mathbb{F}_{-1}$, and $d=\frac{1}{4}(n+4-\sum\limits_{i=1}^{n}a_{i})$. So $$rk(HF^{+}_{red}(M_{T,\phi}))=1,$$ and the d-invariant $$d(M_{T,\phi}, \mathfrak{s}_\xi)=\frac{1}{4}(n+4-\sum\limits_{i=1}^{n}a_{i}).$$ By Proposition~\ref{proposition:Main5}, if $$\frac{1}{4}(n+4-\sum\limits_{i=1}^{n}a_{i})<-2.5,$$ i.e., $$\sum\limits_{i=1}^{n}a_{i}>n+14,$$ then a Stein filling of a contact structure on  $M_{T,\phi}$ with underlying Spin$^{c}$ structure $\mathfrak{s}_\xi$ must be negative-definite.
\end{example}

\section{Existence of Stein fillable contact structures} \label{sec:existence}

\subsection{Sufficiently large Dehn surgeries along  knots} \label{sec:knot}

In this subsection, we show that the sufficiently large Dehn surgeries along several families of knots bound Stein domains. This gives much evidence for Question~\ref{ques:highsurgery} and Conjecture~\ref{conj:highsurgery}.

First, we consider surgeries along twisted torus knots and their connected sums.

\begin{proposition}\label{proposition:twisttorus}
The $r$-surgeries along twisted torus knots $T(u,um+1;u-1,l)$ and $T(u,um+u-1;u-1,l)$ bound Stein domains if $r$ is sufficiently large, where $u\geq3$, $m\geq1$ and $l\geq1$.
\end{proposition}

\begin{figure}[htb]
\begin{overpic}
{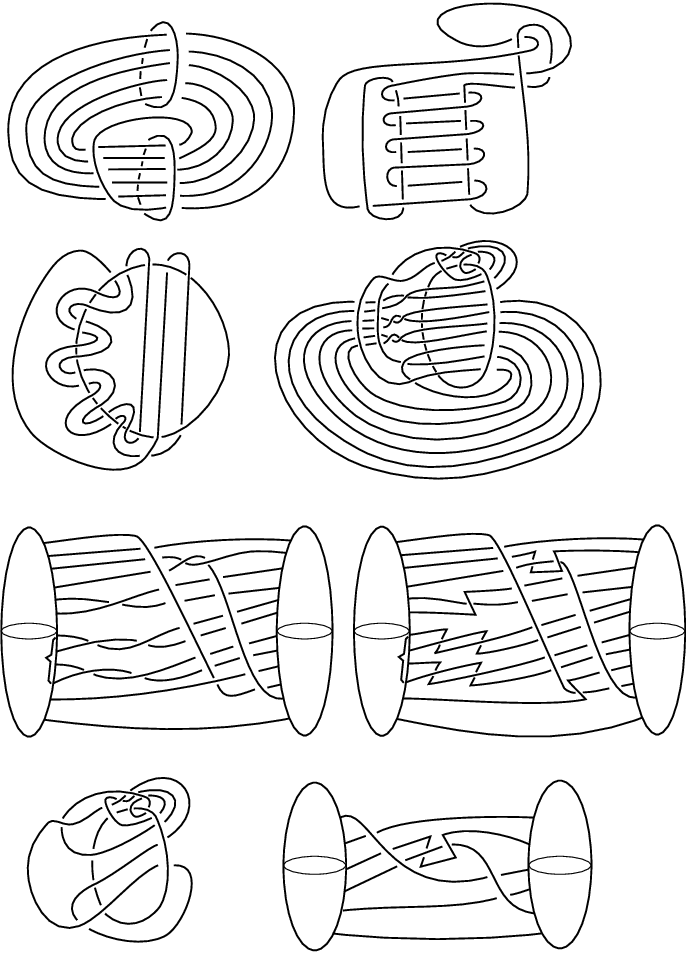}

\put(10, 440){(1)}
\put(280, 440){(2)}

\put(60, 450){$-\frac{1}{l}$}
\put(70, 348){$-\frac{1}{m}$}
\put(0, 420){$r'$}

\put(190, 350){$-\frac{1}{l}$}
\put(260, 380){$-\frac{1}{m}$}
\put(200, 431){$r'$}

\put(10, 330){(3)}
\put(280, 330){(4)}

\put(0, 240){$-\frac{1}{m}$}
\put(50, 285){$-\frac{1}{l}$}
\put(100, 320){$r'$}

\put(10, 210){(5)}
\put(280, 210){(6)}

\put(80, 205){$-\frac{1}{m}$ }
\put(240, 205){$-\frac{1}{m}$ }
\put(78, 118){$-\frac{1}{l}$ }
\put(240, 111){$-\frac{1}{l}$ }

\put(40, 98){$-\frac{1}{r'}$}
\put(210, 96){$-\frac{1}{r'}$}
\put(210, -3){$-\frac{1}{r'}$}

\put(10, 70){(7)}
\put(283, 70){(8)}

\put(180, 70){$-\frac{1}{m}$ }
\put(230, 15){$-\frac{1}{l}$ }

\put(125, 400){$\cdot$ }
\put(127, 400){$\cdot$ }
\put(123, 400){$\cdot$ }

\put(210, 375){$\cdot$ }
\put(210, 373){$\cdot$ }
\put(210, 371){$\cdot$ }

\put(192, 297){$\cdot$ }
\put(192, 301){$\cdot$ }
\put(192, 299){$\cdot$ }

\put(35, 264){$\cdot$ }
\put(35, 268){$\cdot$ }
\put(35, 266){$\cdot$ }

\put(50, 139){$\cdot$ }
\put(50, 142){$\cdot$ }
\put(50, 136){$\cdot$ }

\put(290, 153){$\cdot$ }
\put(290, 150){$\cdot$ }
\put(290, 156){$\cdot$ }
\end{overpic}
\caption{Surgery on twisted torus knot $T(u,um+1;u-1,l)$. Both of the two braces indicate $u-4$ two-strand braids each of which consists of a left-handed full twist.}
\label{figure:bigsur21}
\end{figure}
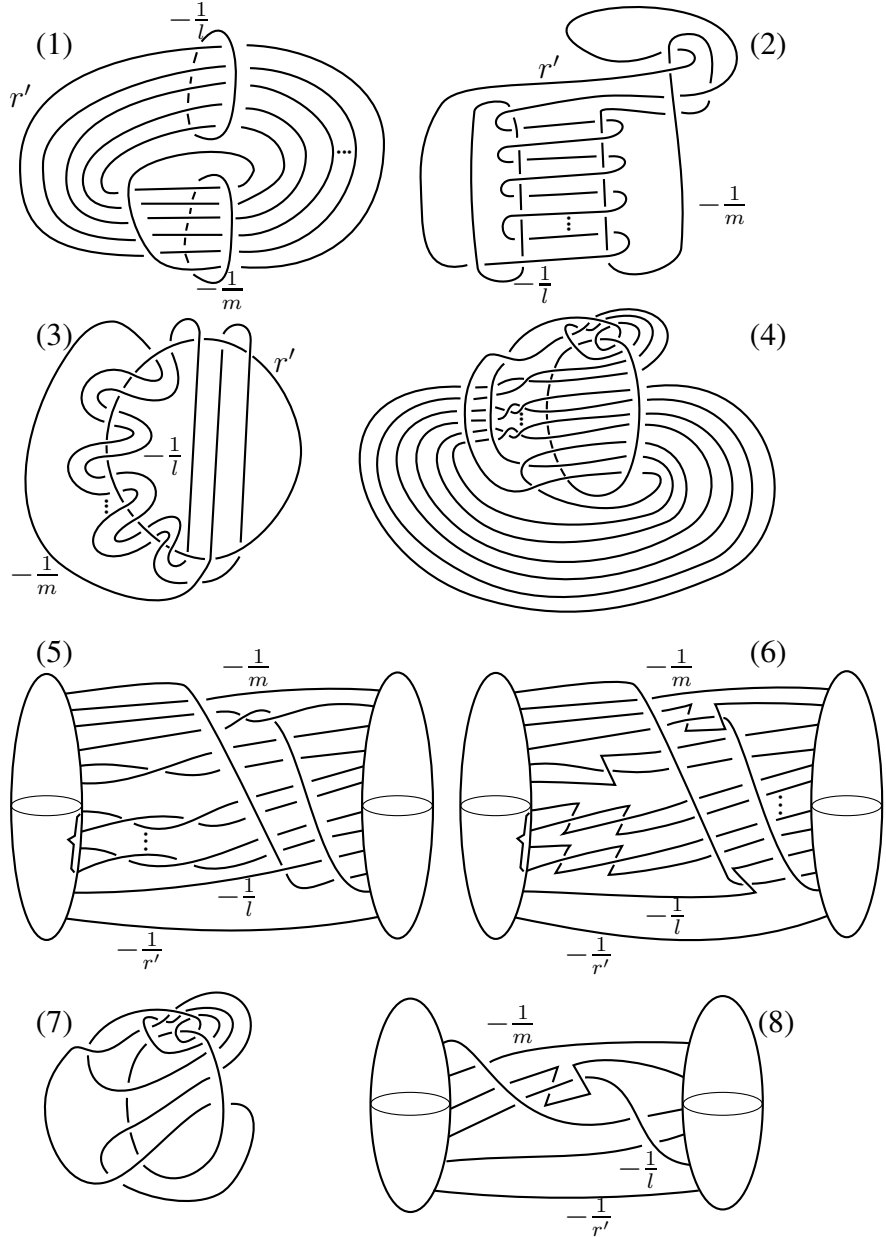

\begin{proof}
The $r$-surgery along $T(u,um+1;u-1,l)$ is equivalent to the $(-\frac{1}{m}, -\frac{1}{l}, r')$-surgery along the link shown in Figure~\ref{figure:bigsur21}-(1), where $r'=r-mu^2-l(u-1)^2$.  If $r$ is sufficiently large, then so is $r'$. By isotopy, we can transform Figure~\ref{figure:bigsur21}-(1) to Figure~\ref{figure:bigsur21}-(3).

If  $u>3$, we can transform the surgery diagram Figure~\ref{figure:bigsur21}-(3) to the surgery diagram Figure~\ref{figure:bigsur21}-(4), and the 4-dimensional handlebody diagram with a one-handle Figure~\ref{figure:bigsur21}-(5). Furthermore, Figure~\ref{figure:bigsur21}-(5) can be transformed to a Stein handlebody diagram Figure~\ref{figure:bigsur21}-(6).
If $u=3$, we can transform the surgery diagram Figure~\ref{figure:bigsur21}-(3) to the surgery diagram Figure~\ref{figure:bigsur21}-(7), and the Stein handlebody diagram Figure~\ref{figure:bigsur21}-(8).

In each of these two Stein handlebody diagrams, there are three Weinstein two-handles. Their attaching spheres are three Legendrian knots. The Thurston-Bennequin invariants are $0$, $1$ and $0$, respectively, all of which are greater than the corresponding smooth surgery coefficients $-\frac{1}{m}$, $-\frac{1}{l}$ and $-\frac{1}{r'}$, respectively.
So the $r$-surgery along $T(u,um+1;u-1,l)$ bounds a Stein domain.

Similarly, the $r$-surgery along $T(u,um+u-1;u-1,l)$ is equivalent to $(-\frac{1}{m}, -\frac{1}{l}, r-mu^2-l(u-1)^2)$-surgery along the link shown in Figure~\ref{figure:bigsur22}-(1). It is equivalent to $(-\frac{1}{m}, -\frac{1}{l+1}, r')$-surgery along the link shown in Figure~\ref{figure:bigsur22}-(2), where $r'=r-mu^2-(l+1)(u-1)^2$.
By isotopy, we can transform Figure~\ref{figure:bigsur22}-(2) to Figure~\ref{figure:bigsur22}-(6). The latter can be transformed to a Stein handlebody diagram Figure~\ref{figure:bigsur22}-(7).

In this Stein handlebody diagram, there are three Weinstein two-handles. Their attaching spheres are three Legendrian knots. The Thurston-Bennequin invariants are all greater than the corresponding smooth surgery coefficients.
So the $r$-surgery along $T(u,um+u-1;u-1,l)$ bounds a Stein domain.
\end{proof}

\begin{figure}[htb]
\begin{overpic}
{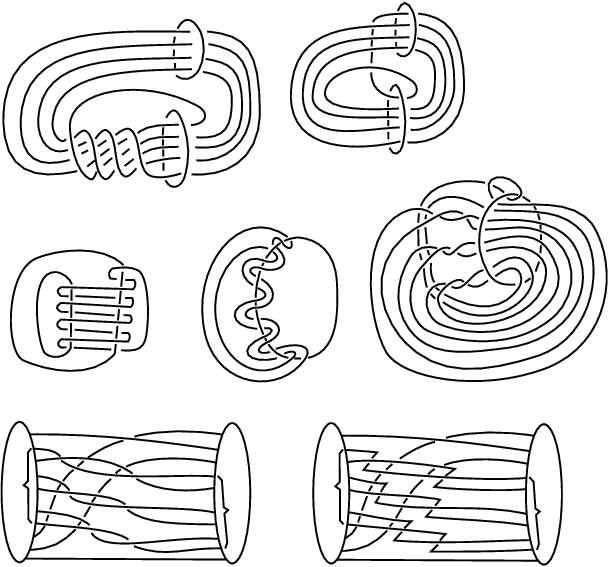}
\put(30, 265){(1)}
\put(85, 265){$-\frac{1}{l}$}
\put(80, 178){$-\frac{1}{m}$ }

\put(160, 265){(2)}
\put(196, 272){$-\frac{1}{l+1}$}
\put(183, 193){$-\frac{1}{m}$ }
\put(140, 250){$r'$}
\put(170, 206){$\cdot$ }
\put(170, 204){$\cdot$ }
\put(170, 202){$\cdot$ }

\put(117, 107){$\cdot$ }
\put(117, 111){$\cdot$ }
\put(117, 109){$\cdot$ }

\put(26, 107){$\cdot$ }
\put(26, 111){$\cdot$ }
\put(26, 109){$\cdot$ }

\put(30, 160){(3)}
\put(0, 155){$-\frac{1}{m}$ }
\put(75, 120){$r'$}

\put(150, 165){(4)}
\put(110, 170){$-\frac{1}{m}$ }
\put(150, 130){$r'$}

\put(280, 180){(5)}
\put(220, 193){$-\frac{1}{m}$ }
\put(220, 146){$\cdot$ }
\put(220, 143){$\cdot$ }
\put(220, 140){$\cdot$ }

\put(30, 70){(6)}
\put(270, 70){(7)}

\put(80, 72){$-\frac{1}{m}$ }

\put(230, 72){$-\frac{1}{m}$ }

\put(260, 180){$r'$}

\put(3, 212){$\cdot$ }
\put(5, 212){$\cdot$ }
\put(7, 212){$\cdot$ }

\put(65, 13){$\cdot$ }
\put(65, 17){$\cdot$ }
\put(65, 15){$\cdot$ }

\put(220, 18){$\cdot$ }
\put(220, 15){$\cdot$ }
\put(220, 12){$\cdot$ }

\put(30, -8){$-\frac{1}{r'}$}
\put(190, -8){$-\frac{1}{r'}$}
\end{overpic}
\caption{Surgery on twisted torus knot $T(u,um+u-1;u-1,l)$. Both of the two braces indicate $u-2$ two-strand braids each of which consists of a left-handed full twist. In the diagrams (3)-(7), the knots without labels all have surgery coefficients $-\frac{1}{l+1}$. }
\label{figure:bigsur22}
\end{figure}

\begin{proposition}\label{proposition:twisttorus1}
Suppose $K_1$, $K_2$, $\cdots$, $K_n$ are twisted torus knots of the form $T(u,um+1;u-1,l)$ or $T(u,um+u-1;u-1,l)$, where $u\geq3$, $m\geq1$ and $l\geq1$. Then $S^{3}_{r}(K_1\# K_2\# \cdots \# K_n)$ bounds a Stein domain for all sufficiently large $r$.
\end{proposition}

\begin{proof}
As shown in Figures~\ref{figure:bigsur21} and ~\ref{figure:bigsur22}, there is a one-handle in the Stein handlebody diagram for the Stein domain bounded by $S^{3}_{r}(K_i)$ for $i=1,2,\cdots, n$. We combine these one-handles to a single one-handle as shown in Figure~\ref{figure:comb1ha}. According to the proof of Proposition~\ref{proposition:twisttorus}, the resulting Stein handlebody fills $S^{3}_{r}(K_1$ $\# K_2$ $\# \cdots \# K_n)$.
\end{proof}

\begin{figure}[htb]
\begin{overpic}
{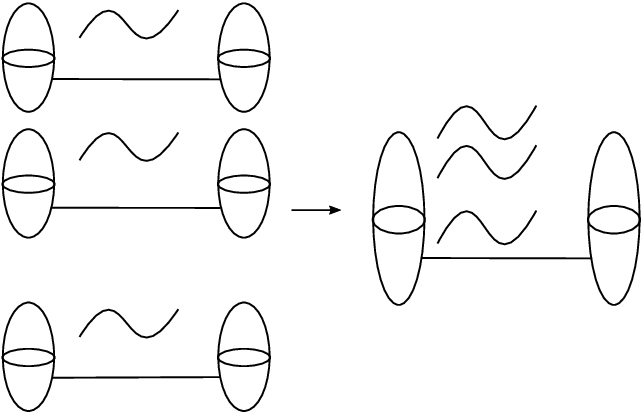}
\put(60, 70){$\vdots$ }
\put(230, 105){$\vdots$ }

\put(250, 60){$-\frac{1}{r'}$ }
\end{overpic}
\caption{Combination of one-handles. The left side consists of $n$ Stein handlebody diagrams as shown in Figures~\ref{figure:bigsur21} and ~\ref{figure:bigsur22}. Each $\sim$ stands for a part of the Stein handlebody diagram. }
\label{figure:comb1ha}
\end{figure}

\begin{proof}[Proof of Theorem~\ref{thm:closed3braids}]
By \cite[Theorem 1]{LeeVa}, a 3-braid L-space knot is either $K_{n,m}$ or $K'_{n,m}$, which corresponds to the twisted torus knot $T(3, 3m+2; 2, n-2)$ or $T(3, 3m+1; 2, n-2)$, respectively.  Thus, the theorem is implied by Proposition~\ref{proposition:twisttorus}
and Proposition~\ref{proposition:twisttorus1}. Moreover, according to the proof of Proposition~\ref{proposition:twisttorus}, the 3-manifold $S^{3}_{r}(K_{n,m})$ (resp. $S^{3}_{r}(K'_{n,m})$) admits a Stein fillable contact structure if $r\geq 9m+4n-4$ (resp. if $r\geq 9m+4n-8$).
\end{proof}

\begin{remark}
Suppose $K$ is a knot in Theorem~\ref{thm:closed3braids}, then Conjecture~\ref{conj:highsurgery} holds for the knot $K$. On the other hand, according to the main result of \cite{eh2}, there exists a Legendrian representative $L$ of $K$ which satisfies the condition of \cite[Theorem 1.1]{ls3} or \cite[Theorem 1.3]{mt}. So one can conclude that  Conjecture~\ref{conj:highsurgery} holds for the knot $K$ via these two theorems.
\end{remark}

\begin{figure}[htb]
\begin{overpic}
{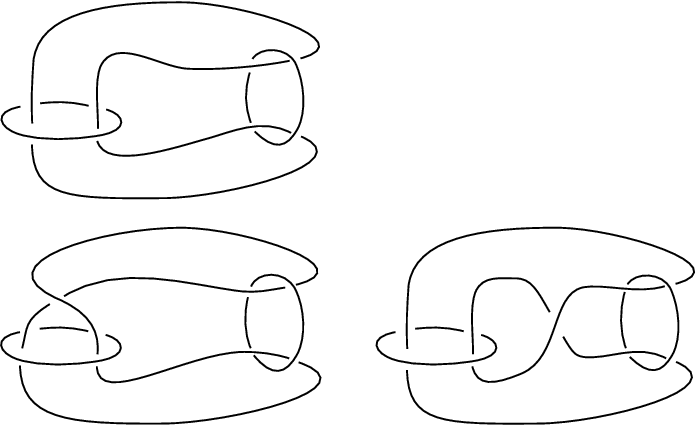}
\put(150, 110){$(1)$ }
\put(0, 162){$\frac{1}{q_1}$ }
\put(148, 160){$\frac{1}{q_2}$ }
\put(10, 190){$r'$ }

\put(150, 0){$(2)$ }
\put(0, 53){$\frac{1}{q_1}$ }
\put(148, 53){$\frac{1}{q_2}$ }
\put(10, 80){$r'$ }

\put(330, 0){$(3)$ }
\put(180, 53){$\frac{1}{q_2}$ }
\put(330, 53){$\frac{1}{q_1}$ }
\put(190, 80){$r'$ }
\end{overpic}
\caption{Surgery on two-bridge knot $K(a_1, a_2)$.  }
\label{figure:2bridge}
\end{figure}

\begin{figure}[htb]
\begin{overpic}
{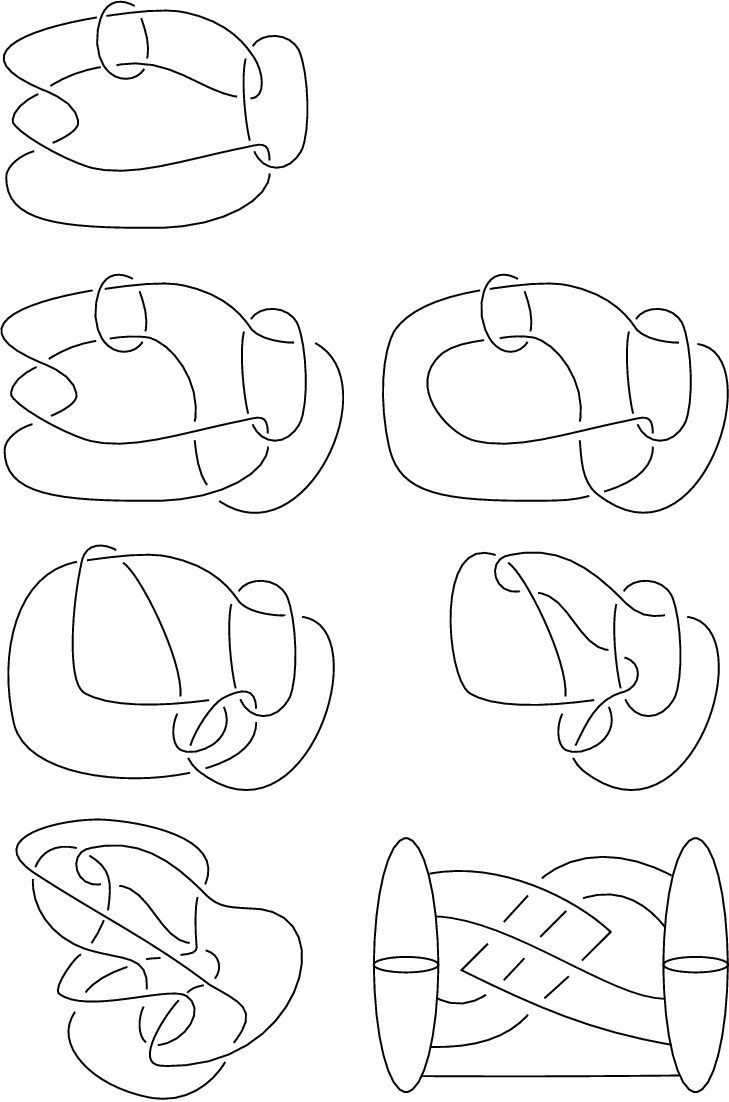}

\put(0, 520){(1)}
\put(62, 482){$\frac{1}{q_{1}}$ }
\put(152, 480){$\frac{1}{q_{2}}$}
\put(120, 420){$r'$ }

\put(0, 390){(2)}
\put(61, 350){$\frac{1}{q_{1}}$ }
\put(145, 385){$\frac{1}{q_{2}}$}
\put(140, 280){$r'$ }

\put(200, 390){(3)}
\put(240, 350){$\frac{1}{q_{1}+1}$ }
\put(330, 385){$\frac{1}{q_{2}+1}$}
\put(330, 280){$r'$ }

\put(0, 260){(4)}
\put(38, 220){$\frac{1}{q_{1}+1}$ }
\put(141, 250){$\frac{1}{q_{2}+1}$}
\put(140, 150){$r'$ }

\put(200, 260){(5)}
\put(220, 220){$\frac{1}{q_{1}+1}$ }
\put(325, 250){$\frac{1}{q_{2}+1}$}
\put(330, 150){$r'$ }

\put(0, 130){(6)}
\put(9, 20){$\frac{1}{q_{2}+1}$ }
\put(0, 80){$\frac{1}{q_{1}+1}$ }
\put(140, 90){$r'$ }

\put(200, 130){(7)}
\put(298, 63){$\frac{1}{q_{2}+1}$ }
\put(228, 120){$\frac{1}{q_{1}+1}$ }
\put(270, -2){$-\frac{1}{r'}$ }
\end{overpic}
\caption{Surgery on two-bridge knot $K(a_1, a_2)$, where $a_1$ is odd and $a_2$ is even.  }
\label{figure:2bridge1}
\end{figure}
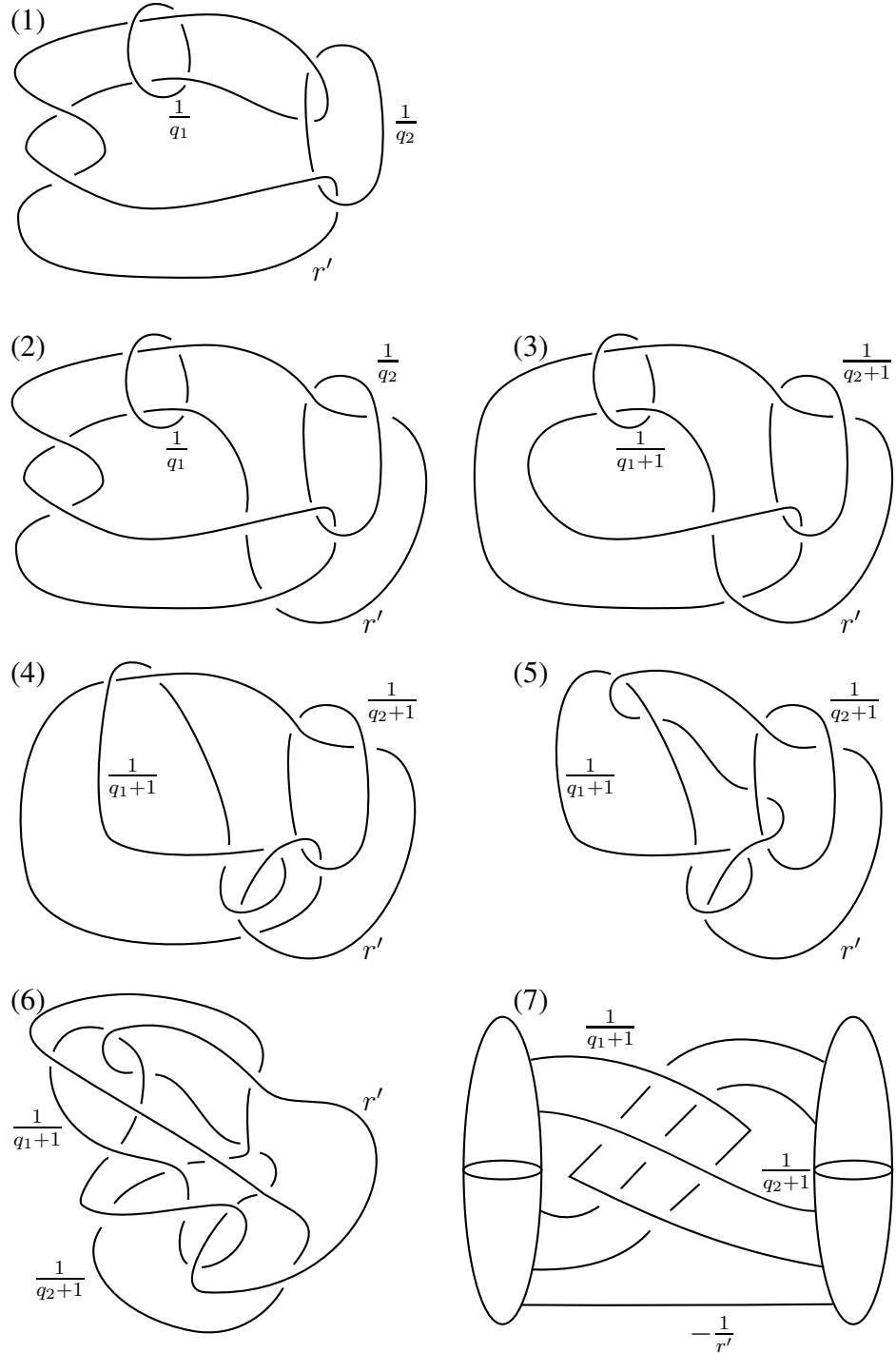

Now we consider surgeries along some two-bridge knots.

\begin{proof}[Proof of Theorem~\ref{thm:2bridgeknot}]
Since $K$ is a knot,  $a_1$ and $a_2$ cannot be both odd.

Case 1, both $a_1$ and $a_2$ are even. Let $a_1=-2q_1$ and $a_2=-2q_2$.
The $r$ surgery along $K$ is equivalent to the a surgery along a Borromean ring with coefficients $\frac{1}{q_1}$, $\frac{1}{q_2}$ and $r$. See Figure~\ref{figure:2bridge}-(1). Since $q_i$ is an integer,  $-1\leq\frac{1}{q_i}\leq1$, for $i=1,2$.

If $\frac{1}{q_1}=\frac{1}{q_2}=1$, and $r>0$. It is an $r$-surgery along the right handed trefoil. This case has been treated in Theorem~\ref{thm:torusknots}. If both $\frac{1}{q_1}$ and $\frac{1}{q_2}$ are positive, either $\frac{1}{q_1}$ or $\frac{1}{q_2}$ is less than $1$, and $r>0$, then it bounds a Stein domain with diagram \cite[Figure 49]{g}. If either $\frac{1}{q_1}$ or $\frac{1}{q_2}$ less than $0$, and $r>0$, then it bounds a Stein domain with diagram \cite[Figure 53]{g}.

Case 2, $a_1$ is odd and $a_2$ is even. Let $a_1=-2q_{1}-1$ and $a_2=-2q_{2}$. The $r$ surgery along $K$ is equivalent to the a surgery along a link shown in Figure~\ref{figure:2bridge}-(2) with coefficients $\frac{1}{q_1}$, $\frac{1}{q_2}$ and $r'=r\pm4q_2$, where $\pm$ is determined by the sign of $q_2$.

Suppose both $q_1$ and $q_2$ are nonzero, and $r$ is sufficiently large. Figure~\ref{figure:2bridge}-(2) can be isotoped to Figure~\ref{figure:2bridge1}-(1). We isotope Figure~\ref{figure:2bridge1}-(1) to Figure~\ref{figure:2bridge1}-(2), and then transform to Figure~\ref{figure:2bridge1}-(3). It can be isotoped to Figure~\ref{figure:2bridge1}-(6), and transformed to a Stein handlebody diagram Figure~\ref{figure:2bridge1}-(7).   In this Stein handlebody diagram, both Legendrian knots have Thurston-Bennequin invariants $1$. The smooth surgery coefficient $\frac{1}{q_{i}+1}$ for the Legendrian knot is either $\infty$ or less than $1$, where $i=1,2$. 


If $q_{2}=0$, then $K$ is the unknot, and the $S^{3}_{r}(K)$ bounds a Stein domain for any $r$.

If $q_{1}=0$, then $K$ is a torus knot which is considered in Theorem~\ref{thm:torusknots}.

Case 3, $a_1$ is even and $a_2$ is odd. Let $a_1=-2q_{1}$ and $a_2=-2q_{2}-1$. The $r$-surgery along $K$ is equivalent to the surgery along the link shown in Figure~\ref{figure:2bridge}-(3) with coefficients $\frac{1}{q_1}$, $\frac{1}{q_2}$ and $r'=r\pm4q_2$. By isotopy, this surgery yields the same 3-manifold as that in Figure~\ref{figure:2bridge}-(2). The argument for the previous case applies.
\end{proof}

\begin{remark}
Suppose $K$ is a twist knot $K(-2q_1, -2)$, where $q_1\leq-3$ is an integer. By Theorem~\ref{thm:2bridgeknot}, Conjecture~\ref{conj:highsurgery} holds for $K$. On the other hand, by \cite[Theorem 1.1]{env}, any Legendrian representative $L$ of $K$ satisfies that $tb(L)-rot(L)\leq q_{1}-1$. Moreover, the genus of $K$ is $1$. So $|\tau(K)|=0$ or $1$, and hence $tb(L)-rot(L)<2\tau(K)-1$. So one cannot conclude that Conjecture~\ref{conj:highsurgery} holds for the twist knot $K$ via \cite[Theorem 1.1]{ls3} or \cite[Theorem 1.2]{mt}.
\end{remark}




\begin{figure}[htb]
\begin{overpic}
{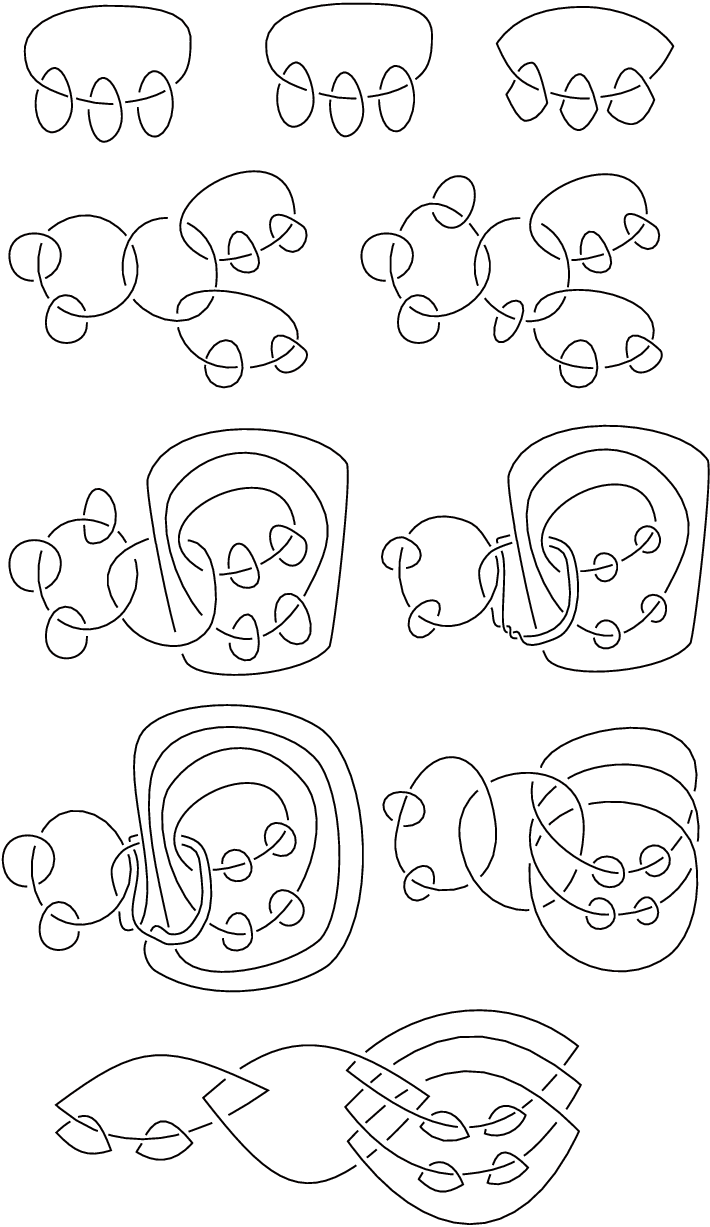}

\put(0, 580){(1)}
\put(50, 580){$1$}
\put(30, 524){$r_1$}
\put(55, 520){$r_2$}
\put(80, 525){$r_3$}

\put(110, 580){(2)}
\put(160, 580){$-2$}
\put(130, 520){$\frac{-r_1}{r_{1}-1}$}
\put(160, 515){$\frac{-r_2}{r_{2}-1}$}
\put(192, 520){$\frac{-r_3}{r_{3}-1}$}

\put(230, 580){(3)}
\put(275, 578){$-2$}
\put(240, 523){$\frac{-r_1}{r_{1}-1}$}
\put(270, 519){$\frac{-r_2}{r_{2}-1}$}
\put(310, 525){$\frac{-r_3}{r_{3}-1}$}

\put(0, 490){(4)}
\put(40, 478){$1$}
\put(120, 498){$1$}
\put(115, 440){$1$}
\put(72, 442){$r'$}
\put(0, 448){$r^{1}_{1}$}
\put(40, 420){$r^{1}_{2}$}
\put(125, 460){$r^{2}_{1}$}
\put(148, 470){$r^{2}_{2}$}
\put(115, 400){$r^{n}_{1}$}
\put(146, 410){$r^{n}_{2}$}
\put(107, 453){$\vdots$}

\put(170, 490){(5)}
\put(277, 455){$\vdots$}
\put(168, 440){$\frac{-r^{1}_1}{r^{1}_{1}-1}$}
\put(190, 410){$\frac{-r^{1}_2}{r^{1}_{2}-1}$}
\put(210, 445){$-2$}
\put(230, 500){$-1$}
\put(240, 475){$2$}
\put(240, 417){$r''$}
\put(280, 495){$-1$}
\put(280, 440){$-1$}
\put(293, 458){$\frac{-r^{2}_1}{r^{2}_{1}-1}$}
\put(320, 475){$\frac{-r^{2}_2}{r^{2}_{2}-1}$}
\put(250, 400){$\frac{-r^{n}_1}{r^{n}_{1}-1}$}
\put(320, 415){$\frac{-r^{n}_2}{r^{n}_{2}-1}$}

\put(0, 360){(6)}
\put(110, 295){$\vdots$}

\put(195, 360){(7)}
\put(282, 300){$\vdots$}
\put(180, 335){$\tilde{r^{1}_1}$}
\put(200, 270){$\tilde{r^{1}_2}$}
\put(285, 330){$\tilde{r^{2}_1}$}
\put(310, 312){$\tilde{r^{2}_2}$}
\put(290, 295){$\tilde{r^{n}_1}$}
\put(310, 280){$\tilde{r^{n}_2}$}
\put(210, 331){$-2$}
\put(224, 310){$2$}
\put(253, 270){$1$}
\put(310, 260){$r''$}
\put(290, 347){$-1$}
\put(290, 364){$-1$}

\put(0, 230){(8)}
\put(110, 155){$\vdots$}

\put(195, 230){(9)}
\put(280, 160){$\vdots$}
\put(180, 213){$\tilde{r^{1}_1}$}
\put(200, 145){$\tilde{r^{1}_2}$}
\put(285, 184){$\tilde{r^{2}_1}$}
\put(309, 188){$\tilde{r^{2}_2}$}
\put(285, 132){$\tilde{r^{n}_1}$}
\put(310, 135){$\tilde{r^{n}_2}$}
\put(310, 115){$r''-1$}
\put(335, 230){$-2$}
\put(335, 195){$-2$}
\put(240, 210){$-2$}
\put(210, 215){$-2$}

\put(0, 70){(10)}
\put(230, 35){$\vdots$}
\put(20, 30){$\tilde{r^{1}_1}$}
\put(70, 18){$\tilde{r^{1}_2}$}
\put(213, 59){$\tilde{r^{2}_1}$}
\put(229, 60){$\tilde{r^{2}_2}$}
\put(216, 8){$\tilde{r^{n}_1}$}
\put(233, 9){$\tilde{r^{n}_2}$}
\put(70, 70){$-2$}
\put(280, 90){$-2$}
\put(280, 70){$-2$}
\put(140, 78){$-2$}
\put(280, 40){$r''-1$}
\end{overpic}
\caption{Surgeries on torus knots and connected sum of torus knots. In (7)-(10), $\tilde{r^{i}_{j}}=\frac{-r^{i}_{j}}{r^{i}_{j}-1}$ for $i=1,2\cdots, n$ and $j=1,2$.}
\label{figure:connec2}
\end{figure}

Next, we consider surgeries along torus knots and their connected sums.
Let $p>q>1$ be two coprime integers, $T(p,q)$ be a positive torus knot, and $T(-p,q)$ be a negative torus knot.

\begin{proof}[Proof of Theorem~\ref{thm:torusknots}]
Case 1, $K=T(p,q)$ is a positive torus knot. By \cite[Lemma 4.4]{OwSt1}, $S^{3}_{r}(T(p,q))$ is a small Seifert fibered space $M(2;-\frac{q^{\ast}}{p}, -\frac{p^{\ast}}{q}, -1-\frac{1}{r-pq}),$  where $qq^{\ast}\equiv 1\pmod p$,  $1\leq q^{\ast}<p$, and $pp^{\ast}\equiv 1\pmod q$, $1\leq p^{\ast}<q$.

We claim that $pq-pp^{\ast}-qq^{\ast}=-1$. Suppose $pp^{\ast}-1=qs$, then $0<s<p$. Since $q(p-s)-1=qp-qs-1=qp-pp^{\ast}$, we have $p|(q(p-s)-1)$. So $q^{\ast}=p-s$ and $pp^{\ast}-1=q(p-q^{\ast})$.

Also, by the classification of Seifert fibered spaces, we have

\begin{align*}
S^{3}_{r}(T(p,q))&=M(1;-\frac{q^{\ast}}{p}, -\frac{p^{\ast}}{q}, -\frac{1}{r-pq})\\
&=M(-1; 1-\frac{q^{\ast}}{p}, 1-\frac{p^{\ast}}{q}, -\frac{1}{r-pq}).
\end{align*}

We denote $r_{1}=\frac{p}{q^{\ast}}$, $r_{2}=\frac{q}{p^{\ast}}$, and $r_{3}=r-pq$, then
\begin{align*}
 S^{3}_{r}(T(p,q))&=M(1;-\frac{1}{r_1}, -\frac{1}{r_2}, -\frac{1}{r_3})\\
 &=M(-1; \frac{r_{1}-1}{r_{1}}, \frac{r_{2}-1}{r_{2}}, -\frac{1}{r_{3}})\\
 &=M(-2; \frac{r_{1}-1}{r_{1}}, \frac{r_{2}-1}{r_{2}}, \frac{r_{3}-1}{r_{3}}).
\end{align*}

See Figure~\ref{figure:connec2}-(1) and Figure~\ref{figure:connec2}-(2).
Note that $r_1, r_2>1$. If $r>pq+1$, then $\frac{-r_i}{r_{i}-1}<-1$ for $i=1,2,3$. So $S^{3}_{r}(K)$ bounds a Stein domain if $r>pq+1$. The Stein domain has a handlebody diagram shown in  Figure~\ref{figure:connec2}-(3).

Case 2, $K=T(-p,q)$ is a negative torus knot. Then
\begin{align*}
S^{3}_{r}(T(-p,q))&=-S^{3}_{-r}(T(p,q))\\
&=-M(2;-\frac{q^{\ast}}{p}, -\frac{p^{\ast}}{q}, -1-\frac{1}{-r-pq})\\
&=M(-2;\frac{q^{\ast}}{p}, \frac{p^{\ast}}{q}, 1-\frac{1}{r+pq})\\
&=M(1; -\frac{p-q^{\ast}}{p}, -\frac{q-p^{\ast}}{q}, -\frac{1}{r+pq}).
\end{align*}

We denote $r_{1}=\frac{p}{p-q^{\ast}}$, $r_{2}=\frac{q}{q-p^{\ast}}$, and $r_{3}=r+pq$. Then
\begin{align*}
S^{3}_{r}(T(-p,q))&=M(1;-\frac{1}{r_1}, -\frac{1}{r_2}, -\frac{1}{r_3})\\
&=M(-1; \frac{r_{1}-1}{r_{1}}, \frac{r_{2}-1}{r_{2}}, -\frac{1}{r_{3}})\\
&=M(-2; \frac{r_{1}-1}{r_{1}}, \frac{r_{2}-1}{r_{2}}, \frac{r_{3}-1}{r_{3}}).
\end{align*}
See Figure~\ref{figure:connec2}-(1) and Figure~\ref{figure:connec2}-(2) again.
Note that $r_1,r_2>1$. If $r>-pq+1$, then $\frac{-r_i}{r_{i}-1}<-1$ for $i=1,2,3$. So $S^{3}_{r}(K)$ bounds a Stein domain if $r>-pq+1$. The Stein domain also has a handlebody diagram, as shown in  Figure~\ref{figure:connec2}-(3).  

Case 3, $K=T(u_1,v_1)\# T(u_2,v_2)\#\cdots
\# T(u_n,v_n)$ is the connected sum of $n$ torus knots. Here we assume that $|u_i|>v_i>1$.

Suppose that $S^{3}_{r}(T(u_i,v_i))$ has a surgery diagram shown in Figure~\ref{figure:connec2}-(1) with surgery coefficients $1$, $r^{i}_1$, $r^{i}_2$ and $r^{i}_3$, where $i=1,2,\cdots, n$. We claim that the $r$-surgery along $K$, $S^{3}_{r}(K)$, has a surgery diagram as shown in Figure~\ref{figure:connec2}-(4), where 
$r'=r-u_{1}v_{1}-u_{2}v_{2}-\cdots-u_{n}v_{n}$. This can be shown as follows. The surgered manifold $S^{3}_{r}(T(p,q))$ has a surgery diagram Figure~\ref{figure:connec2}-(1). If we remove the component in Figure~\ref{figure:connec2}-(1) with surgery coefficient $r_{3}=r-pq$, then the remaining surgery diagram yields the 3-sphere. In fact, using slam-dunk, the remaining surgery diagram is a $(\frac{p}{q^{\ast}},\frac{q-p^{\ast}}{q})$-surgery along the Hopf link. Since $pq-pp^{\ast}-qq^{\ast}=-1$, this surgery yields a 3-sphere. So this operation reduces the surgery diagram in Figure~\ref{figure:connec2}-(1) to a $\hat{r}$-surgery along some knot $\hat{K}\subset S^{3}$, where $\hat{r}$ and $r-pq$ differ by some integer. For homological reasons, $\hat{r}=r$.  By \cite[Theorem 1.3]{nz}, any sufficiently large rational is a characterizing slope for the knot $T(p,q)$. So $\hat{K}=T(p,q)$. The same thing holds for a negative torus knot $T(-p,q)$. Therefore, the surgery diagrams consist of the three-component chain links with surgery coefficients $1$, $r^{i}_1$ and $r^{i}_2$ in Figure~\ref{figure:connec2}-(4) all yield the 3-sphere, and these operations together reduce Figure~\ref{figure:connec2}-(4) to the $r$-surgery along the composite knot $T(u_1,v_1)\# T(u_2,v_2)\#\cdots
\# T(u_n,v_n)$.

Recall that $M(-2; \frac{r_{1}-1}{r_{1}}, \frac{r_{2}-1}{r_{2}}, \frac{r_{3}-1}{r_{3}})=M(-1; \frac{r_{1}-1}{r_{1}}, \frac{r_{2}-1}{r_{2}}, -\frac{1}{r_{3}})$. Let $r''$ satisfy $r'=2-\frac{1}{r''}$.  Figure~\ref{figure:connec2}-(4) can be transformed to Figure~\ref{figure:connec2}-(5) by two slam-dunks. If $r$ is sufficiently large, then $r''<0$. We isotope Figure~\ref{figure:connec2}-(5) to Figure~\ref{figure:connec2}-(6). We slide a two-handle with framing $-1$ along a two-handle with framing $2$, and get Figure~\ref{figure:connec2}-(7). We isotope Figure~\ref{figure:connec2}-(7) to Figure~\ref{figure:connec2}-(8). We blow down the two-handle with framing $1$, and obtain Figure~\ref{figure:connec2}-(9).  This can be transformed to a Stein handlebody diagram as shown in Figure~\ref{figure:connec2}-(10). In this Stein handlebody diagram, the attaching Legendrian knot of each Weinstein two-handle has Thurston-Bennequin invariant $-1$, and the smooth surgery coefficients are all smaller than $-1$. So $S^{3}_{r}(K)$ bounds a Stein domain.
\end{proof}

At last, we consider surgeries along some pretzel knots.

\begin{proof}[Proof of Theorem~\ref{thm:pretzel}]
The $r$-surgery along $P(-2l-1,2m+1,2n+1)$ is equivalent to $(\frac{1}{l}, -\frac{1}{m+1}, -\frac{1}{n}, r)$-surgery along the link shown in Figure~\ref{figure:P(-2l-1)}-(1). We isotope Figure~\ref{figure:P(-2l-1)}-(1) to Figure~\ref{figure:P(-2l-1)}-(4), and represent it by the boundary of a handlebody in Figure~\ref{figure:P(-2l-1)}-(7). We transform the handlebody to a Stein handlebody in Figure~\ref{figure:P(-2l-1)}-(8). In the Stein handlebody diagram, both the Thurston-Bennequin invariants are $0$, and $ -\frac{1}{m+1}, -\frac{1}{n}<0$.
\end{proof}

\begin{figure}[htb]
\begin{overpic}
{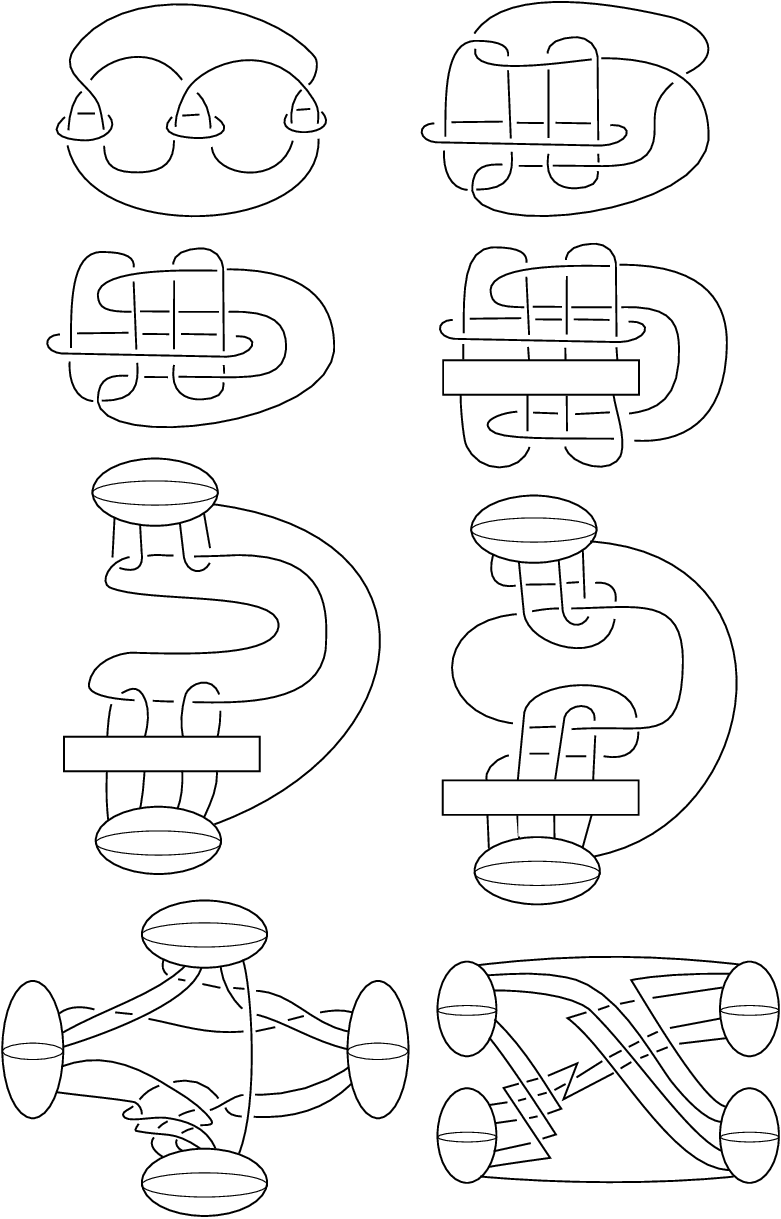}

\put(0, 580){(1)}
\put(130, 580){$r$ }
\put(-2, 525){$-\frac{1}{m+1}$ }
\put(62, 525){$-\frac{1}{n}$ }
\put(160, 525){$\frac{1}{l}$ }

\put(200, 580){(2)}
\put(330, 580){$r$ }
\put(185, 546){$-\frac{1}{m+1}$ }
\put(290, 544){$-\frac{1}{n}$ }
\put(195, 520){$\frac{1}{l}$ }

\put(0, 460){(3)}
\put(150, 420){$r$ }
\put(15, 420){$\frac{1}{l}$ }
\put(15, 470){$-\frac{1}{m+1}$ }
\put(100, 466){$-\frac{1}{n}$ }

\put(200, 460){(4)}
\put(340, 420){$r$ }
\put(205, 430){$\frac{1}{l}$ }
\put(200, 360){$-\frac{1}{m+1}$ }
\put(295, 360){$-\frac{1}{n}$ }

\put(0, 330){(5)}
\put(100, 327){$-\frac{1}{n}$ }
\put(25, 325){$-\frac{1}{m+1}$ }
\put(145, 280){$r$ }
\put(142, 200){$-l$ }

\put(200, 330){(6)}
\put(210, 300){$-\frac{1}{m+1}$ }
\put(320, 260){$r$ }
\put(320, 180){$-l$ }
\put(282, 183){$-\frac{1}{n}$ }

\put(0, 130){(7)}
\put(150, 106){$-\frac{1}{r}$ }
\put(128, 72){$-\frac{1}{m+1}$ }
\put(150, 47){$-\frac{1}{n}$ }
\put(120, 35){$-l$ }

\put(200, 130){(8)}
\put(320, 8){$-l$ }
\put(320, 132){$-\frac{1}{r}$ }
\put(330, 70){$-\frac{1}{m+1}$ }
\put(305, 33){$-\frac{1}{n}$ }

\end{overpic}
\caption{Surgery on the pretzel knot $P(-2l-1,2m+1,2n+1)$. Each box represents a left-handed half twist.  }
\label{figure:P(-2l-1)}
\end{figure}


\subsection{Stein fillable coefficients}
Question~\ref{ques:highsurgery} is equivalent to the following question.

\begin{question}
Suppose $K$ is a knot in $S^3$, is $Sfc(K)<\infty?$
\end{question}

Obviously, the Stein fillable  coefficient of the unknot is its maximal Thurston-Bennequin invariant $-1$.

Suppose $K=T(p,q)$ is a positive torus knot, where $p>q>1$.  Consider the continued fraction expansion
\begin{equation*}
\frac{p}{q}=c_1+\cfrac{1}{c_2+\cfrac{1}{\ddots\cfrac{\ddots}{c_{n-1}+\cfrac{1}{c_n}}}},
\end{equation*}
where $c_{i}\geq2$.
Then \cite[Theorem 2]{OwSt1} and the last paragraph of \cite{OwSt1} imply that $$Sfc(K)=m(T(p,q))=\begin{cases}
pq-\frac{q}{p^{\ast}}& \text{if}~n~ \text{is even},\\
pq-\frac{p}{q^{\ast}}& \text{if}~n~ \text{is odd}.
\end{cases}$$


\begin{proof}[Proof of Proposition~\ref{sfcandm}]
If $r$ is a positive rational and $r\geq Sfc(K)$, then $S^{3}_{r}(K)$ is an L-space and bounds a Stein domain $W$. By Theorem~\ref{thm: negdefinite}, $W$ must be negative-definite. So  $r\geq m(K)$. Therefore, $Sfc(K)\geq m(K)$.
\end{proof}

\begin{example}\label{example:torusknot}
If $K$ is a negative torus knot $T(-p,q)$, then $Sfc(K)=-pq$.
\end{example}

\begin{proof}

In \cite{LeLi}, Lecuona and Lisca classified all Seifert fibered spaces that are Stein fillable. In particular, we apply their result to the manifold $$S^{3}_{r}(T(-p,q))=M(-2;\frac{q^{\ast}}{p}, \frac{p^{\ast}}{q}, 1-\frac{1}{r+pq}).$$ 


If $r\in (-pq+1, \infty)$, then $1-\frac{1}{r+pq}\in(0,1)$; \cite[Theorem 1.5] {LeLi} implies that $S^{3}_{r}(T(-p,q))$ bounds a Stein domain. If $r= -pq+1$, then $S^{3}_{r}(T(-p,q))$ is a lens space which bounds a Stein domain. If $r\in(-pq, -pq+1)$, then $1-\frac{1}{r+pq}<0$; again, \cite[Theorem 1.5] {LeLi} implies that $S^{3}_{r}(T(-p,q))$ bounds a Stein domain.  If $r=-pq$, then $S^{3}_{r}(T(-p,q))$ is a connected sum of two lens spaces which bounds a Stein domain. So $Sfc(T(-p,q))\leq -pq$, which is the maximal Thurston-Bennequin invariant of $T(-p,q)$. By definition, $Sfc(T(-p,q))=-pq$.
\end{proof}

\begin{example}
Let $K$ be the two-bridge knot $K(-2q_{1}, -2q_{2})$, where $q_{1}, q_{2}>0$, then $Sfc(K(-2q_{1}, -2q_{2}))=1$.
\end{example}

\begin{proof}
It follows from the proof of  Proposition~\ref{thm:2bridgeknot} that $S^{3}_{r}(K)$ bounds a Stein domain for $r\geq1$. On the other hand, it is easy to find a Legendrian representative of $K$ with Thurston-Bennequin invariant $1$. Moreover, the genus of $K$ is $1$. So the maximal Thurston-Bennequin invariant $TB(K(-2q_{1}, -2q_{2}))=1$. Thus $Sfc(K(-2q_{1}, -2q_{2}))=1$.
\end{proof}

\subsection{Sufficiently large Dehn surgeries along links} \label{sec:link}
In this subsection, we give some evidence for Question~\ref{ques:link}.

Any surgery along the Hopf link yields a lens space which is known to bound a Stein domain.
According to \cite[Figure 52]{g}, the $(r_1,r_2,r_3)$-surgery along the Borromean ring bounds a Stein domain if $r_1\geq 1$, $r_2>0$ and $r_3\geq 4$.  Using \cite[Figures 49 and 52]{g}, one can show that the $(r_1,r_2)$-surgery along the Whitehead link (and its mirror image) bounds a Stein domain if $r_1, r_2$ are sufficiently large. This is because it is equivalent to some surgery along the Borromean ring.

\begin{proof}[Proof of Theorem~\ref{thm:biglink}]
The $(r_1,r_2)$-surgery along $\mathbb{L}_n$ is equivalent to the $(r_1, r'_{2}, -\frac{1}{n+1})$-surgery along the link in Figure~\ref{figure:BigLn}-(1), where $r'_{2}=r_{2}-4n-4$. We isotope Figure~\ref{figure:BigLn}-(1) to Figure~\ref{figure:BigLn}-(5). It bounds a 4-dimensional handlebody diagram Figure~\ref{figure:BigLn}-(6).  We isotope Figure~\ref{figure:BigLn}-(6) to Figure~\ref{figure:BigLn}-(7). By wrapping one strand around the upper left attaching ball,  we can transform Figure~\ref{figure:BigLn}-(7) to a Stein handlebody diagram Figure~\ref{figure:BigLn}-(8).  The Thurston-Bennequin invariant is $0$, and the smooth surgery coefficient is $-\frac{1}{n+1}<0$.
\end{proof}

\begin{figure}[htb]
\begin{overpic}
{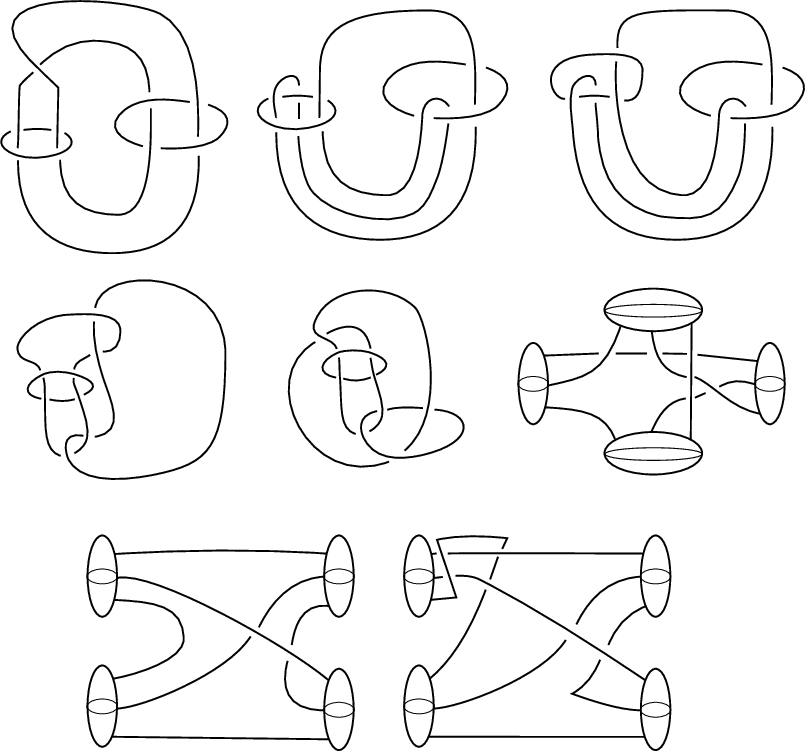}
\put(0, 360){(1)}
\put(150, 360){(2)}
\put(290, 360){(3)}
\put(90, 250){$r_{2}'$ }
\put(50, 314){$r_{1}$ }
\put(-27, 295){$-\frac{1}{n+1}$}

\put(0, 210){(4)}
\put(135, 210){(5)}
\put(272, 210){(6)}
\put(218, 250){$r_{2}'$ }
\put(180, 332){$r_{1}$ }

\put(360, 250){$r_{2}'$ }
\put(320, 330){$r_{1}$ }

\put(96, 132){$r_{2}'$ }
\put(0, 170){$r_{1}$ }

\put(150, 172){$r_{1}$ }
\put(220, 139){$r_{2}'$ }

\put(345, 200){$-\frac{1}{r_{1}}$ }
\put(333, 158){$-\frac{1}{r_{2}'}$ }

\put(120, 105){$-\frac{1}{r_{1}}$ }
\put(90, 17){$-\frac{1}{r_{2}'}$ }

\put(280, 105){$-\frac{1}{r_{1}}$ }
\put(240, 17){$-\frac{1}{r_{2}'}$ }
\put(188, 50){$-\frac{1}{n+1}$}

\put(24, 100){(7)}
\put(330, 100){(8)}
\end{overpic}
\caption{Surgery on the link $\mathbb{L}_n$. In diagrams (2)-(7), the knots without labels all have surgery coefficients $-\frac{1}{n+1}$. }
\label{figure:BigLn}
\end{figure}

\begin{proof}[Proof of Theorem~\ref{thm:biglink1}]

We assume that $a_{i}=2q_{i}+1$ for $i=1,2$. Then the $(r_1, r_2)$-surgery along $\mathbb{L}$ is equivalent to the surgery diagram in Figure~\ref{figure:BigLn1}-(1), where $r'_{i}=r_{i}-q_{1}-q_{2}$ for $i=1,2$. By isotopy, we obtain the surgery diagram in Figure~\ref{figure:BigLn1}-(2).

If either $q_1$ or $q_2$ is $-1$ or $0$, then the $(r_1, r_2)$-surgery along $\mathbb{L}$ yields a lens space, a connected sum of two lens space, or a small Seifert fibered space. The former two cases certainly admit Stein fillable contact structures. By the main result in \cite{LeLi}, such small Seifert fibered spaces admit Stein fillable contact structures.

Suppose both $q_1$ and $q_2$ are smaller than $-1$.  Performing isotopy and slam-dunk, we obtain the surgery diagram in Figure~\ref{figure:BigLn1}-(4), where $2-\frac{1}{r_{1}''}=r_{1}'$. Since $r_1$ is sufficiently large, $r_{1}''<0$. Similar to the proof of Theorem~\ref{thm:torusknots}, we slide a 2-handle with framing $-1$ to the 2-handle with framing $2$, and by blowing down a resulting 2-handle with framing $1$, we obtain the surgery diagram in Figure~\ref{figure:BigLn1}-(5). By slam-dunk, we obtain the surgery diagram in Figure~\ref{figure:BigLn1}-(6), where where $2-\frac{1}{r_{2}''}=r_{2}'$ and $r_{2}''<0$. Using the same operation as before, we obtain the surgery diagram in Figure~\ref{figure:BigLn1}-(7). It can be adapted to a Stein handlebody diagram as in Figure~\ref{figure:BigLn1}-(8).

Now we consider the remaining cases. Suppose for example $q_{1}>0$.  Let $-\frac{1}{q_1}=-1-\frac{1}{-\frac{q_1}{q_{1}-1}}$. Repeating the previous argument with a slight modification, we can obtain a Stein handlebody diagram.  In this case, the surgery diagram in Figure~\ref{figure:BigLn1}-(3) has different coefficients. The left unknot with coefficient $0$ should have coefficient $-1$. The unknot with coefficient $q_1$ should have coefficient $-\frac{q_1}{q_{1}-1}$. The coefficients in other surgery diagrams are modified accordingly.
\end{proof}

\begin{figure}[htb]
\begin{overpic}
{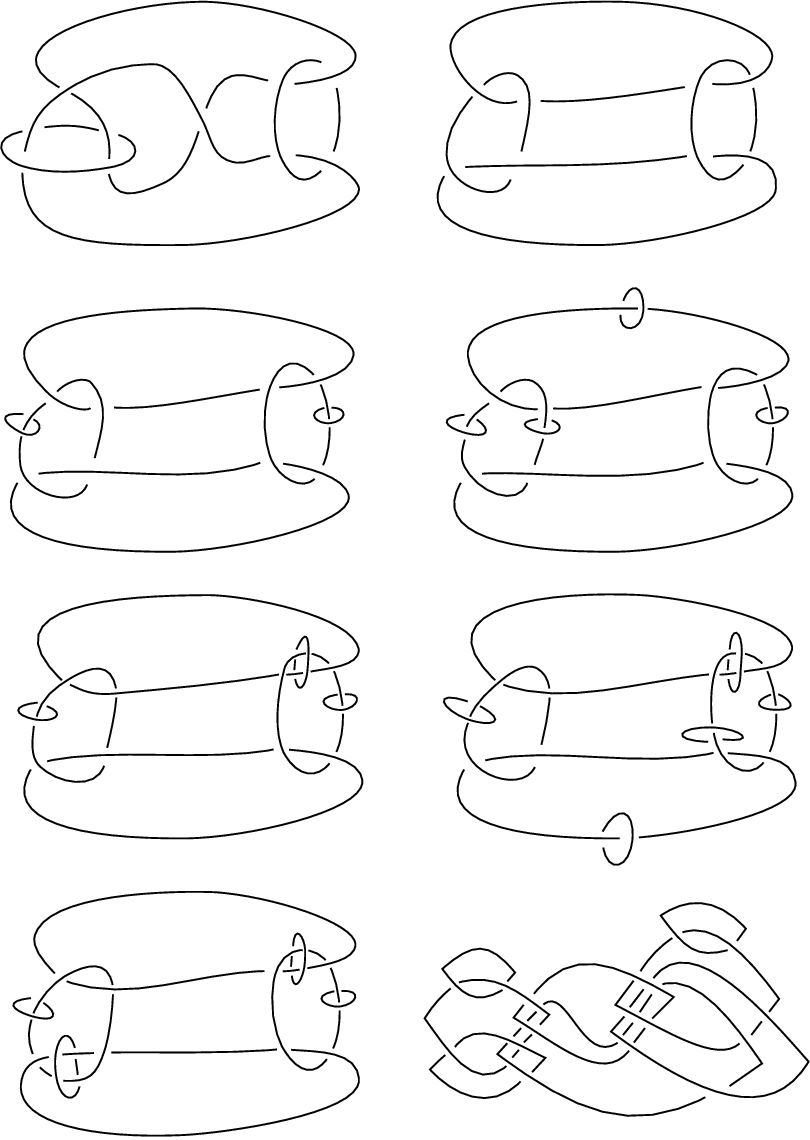}
\put(0, 540){(1)}
\put(150, 540){$r_{1}'$ }
\put(150, 430){$r_{2}'$ }
\put(-7, 493){$-\frac{1}{q_1}$}
\put(165, 490){$-\frac{1}{q_2}$}

\put(210, 540){(2)}
\put(354, 540){$r_{1}'$ }
\put(350, 430){$r_{2}'$ }
\put(200, 494){$-\frac{1}{q_1}$}
\put(365, 490){$-\frac{1}{q_2}$}

\put(0, 400){(3)}
\put(40, 370){$0$}
\put(140, 305){$0$}
\put(0, 355){$q_1$ }
\put(170, 350){$q_2$ }
\put(150, 288){$r_{2}'$ }
\put(150, 395){$r_{1}'$ }

\put(210, 400){(4)}
\put(250, 370){$-1$}
\put(350, 305){$0$}
\put(270, 335){$-1$}
\put(210, 353){$q_1$ }
\put(382, 350){$q_2$ }
\put(360, 288){$r_{2}'$ }
\put(310, 410){$r_{1}''$ }
\put(360, 395){$2$ }

\put(0, 260){(5)}
\put(-1, 210){$q_1$ }
\put(174, 210){$q_2$ }
\put(110, 240){$r_{1}''-1$ }
\put(150, 255){$-2$ }
\put(50, 230){$-1$ }
\put(140, 167){$-1$ }
\put(150, 145){$r_{2}'$ }

\put(210, 260){(6)}
\put(250, 230){$-1$}
\put(350, 168){$-2$}
\put(314, 192){$-1$}
\put(320, 240){$r_{1}''-1$ }
\put(302, 130){$r_{2}''$ }
\put(210, 217){$q_1$ }
\put(383, 210){$q_2$ }
\put(360, 145){$2$ }
\put(360, 255){$-2$ }

\put(0, 110){(7)}
\put(0, 70){$q_1$ }
\put(172, 70){$q_2$ }
\put(110, 95){$r''_{1}-1$}
\put(40, 15){$r''_{2}-1$}
\put(150, 4){$-2$}
\put(150, 110){$-2$}
\put(48, 83){$-2$}
\put(140, 25){$-2$}

\put(210, 110){(8)}
\put(240, 92){$q_1$ }
\put(350, 113){$q_2$ }
\put(360, 10){$r''_{1}-1$}
\put(210, 10){$r''_{2}-1$}
\put(280, 0){$-2$}
\put(188, 50){$-2$}
\put(375, 70){$-2$}
\put(280, 88){$-2$}

\end{overpic}
\caption{Surgery on the link $\mathbb{L}$.  }
\label{figure:BigLn1}
\end{figure}

If the answer to Question~\ref{ques:link} is true, then by Proposition~\ref{char1} the $(r_{1}, r_{2},\cdots, r_{n})$-surgery along a link $K_{1}\cup K_{2}\cdots\cup K_{n}\subset S^3$ yields a 3-manifold which admits a Stein fillable contact structure as long as each $r_{i}$ is either sufficiently large or sufficiently small (or negative), where $i=1,2,\cdots,n$.

However, if there are some $r_{i}$ that are neither sufficiently large nor sufficiently negative, then the surgery may does not exist a symplectic fillable contact structure. For example,  for any integer $n\geq2$, the $(-2n+3, -\frac{1}{n})$-surgery along the link shown in Figure~\ref{figure:plink1} does not admit a symplectic fillable contact structure, since it is equivalent to the $(2n+3)$-surgery along the pretzel knot $P(-2,3,2n+1)$ and Theorem~\ref{Theorem:Main1} applies. Here one coefficient can be arbitrarily negative, while the other coefficient lies in the interval $(-1,0)$.

\begin{figure}[htb]
\begin{overpic}
{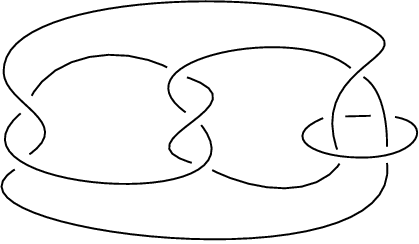}
\put(170, 110){$-2n+3$ }
\put(200, 60){$-\frac{1}{n}$}

\end{overpic}
\caption{Surgery on a link.  }
\label{figure:plink1}
\end{figure}

\end{document}